\documentclass[12pt,a4paper,oneside,reqno]{amsart}
\title[]{The Chern-Weil homomorphism for deformed Hopf-Galois extensions}
\author{Jacopo Zanchettin}
\address{\textit{SISSA, Via Bonomea 265, 34136 Trieste, Italy}}
\email{jzanchet@sissa.it}
\date{}
\usepackage{amsmath,amssymb,amstext,amsthm,inputenc,graphicx,charter,authblk,tensor,tikz-cd,verbatim}
\usepackage[top=2.5cm,bottom=2.5cm,left=2.5cm,right=2.5cm,heightrounded,bindingoffset=10mm]{geometry}
\usepackage{color}

\usepackage{hyperref}
\hypersetup{colorlinks,
	linkcolor=black!75!red,
	citecolor=blue,
	pdftitle={},
	pdfproducer={pdfLaTeX},
	pdfpagemode=None,
	bookmarksopen=true
	bookmarksnumbered=true}
\usepackage{slashed}
\numberwithin{equation}{section}
\theoremstyle{plain} 
\newtheorem{thm}{Theorem}[section] 
\newtheorem{cor}[thm]{Corollary} 
 
\newtheorem{prop}[thm]{Proposition} 
\newtheorem{defn}[thm]{Definition}
\newtheorem{expl}[thm]{Example}
\newtheorem{remark}[thm]{Remark}
\DeclareMathOperator{\id}{\mathrm{id}}

\DeclareMathOperator{\Hom}{\mathrm{Hom}}
\DeclareMathOperator{\can}{\mathrm{can}}
\DeclareMathOperator{\chw}{\mathrm{chw}}
\newcommand{\1}{\langle 1 \rangle}
\newcommand{\2}{\langle 2 \rangle}
\newcommand{\As}{[A\square^H-]}

\newcommand{\beq}{\begin{equation}}
	\newcommand{\eeq}{\end{equation}}
\newcommand{\beqn}{\begin{equation*}}
	\newcommand{\eeqn}{\end{equation*}}

\newcommand\bK{{\mathbb K}}

\newcommand\bN{{\mathbb N}}

\newcommand\bR{{\mathbb R}}

\newcommand\bZ{{\mathbb Z}}

\newcommand\cA{{\mathcal A}}
\newcommand\cB{{\mathcal B}}

\newcommand\cQ{{\mathcal Q}}

\newcommand\fM{{\mathfrak M}}

\newcommand\fS{{\mathfrak S}}

\newcommand\ff{{\mathfrak f}}
\newcommand\fg{{\mathfrak g}}

\begin{document}
	\maketitle
	
	\begin{abstract}
		In this article, we study the Chern-Weil theory for Hopf-Galois extensions originally introduced by Hajac and Maszczyk in the context of coalgebra extensions. We show that the cyclic homology Chern-Weil homomorphism defines natural transformations between Hopf-Galois extensions with a strong connection (principal comodule algebras) and cyclic homology, thereby generalizing the concept of characteristic classes to the noncommutative setting. In the second part, we study the effect of $2$-cocycle deformations of Hopf-Galois extensions on the aforementioned homomorphism. We consider the $2$-cocycle coming from the structure Hopf algebra of the extension, an external symmetry, and finally the combined case.
	\end{abstract}
	
	\tableofcontents
	
	\section*{Introduction}
	The classical Chern-Weil homomorphism\footnote{We refer to the standard textbook \cite{kobayashi1996foundations} and references therein, including the original one from Cartan, for an extensive treatment of this subject.} allows one to find characteristic classes of principal and vector bundles. The latter are global invariants of the base manifold $X$ of the bundle and they can be used to prove the non-triviality of the bundle: if some class is non-trivial, so is the bundle. This is a consequence of the fact that all characteristic classes of a trivial bundle are trivial. The Chern-Weil homomorphism computes characteristic classes in terms of curvature form of connections and takes values in the de Rham cohomology of the base manifold $H_{dR}^*(X)$. Given a principal $G$-bundle on $X$ ($G$ is a Lie group), the curvature form of a connection $\nabla$ is a Lie algebra-valued $2$-form $\Omega_{\nabla}$. Let $\bK$ be a field, if we denote by $\bK[\fg]^G$ ring of adjoint-invariant $\bK$-valued polynomials of the Lie algebra $\fg$, the Chern-Weil homomorphism is well-defined algebra morphism
	\begin{equation*}
		\mathrm{Chw}:\bK[\fg]^G\longrightarrow H_{dR}^{\textit{even}}(X),\quad p\longmapsto \left[p(\Omega_{\nabla})\right].
	\end{equation*}

	In the context of noncommutative geometry, Hajac and Maszczyk developed the Chern-Weil theory for coalgebra extensions \cite{hajac2021cyclic}. In this paper we focus on a particular case of such extensions, those of Hopf-Galois. Here we have a Hopf algebra $H$ (co)acting on an algebra $A$ with (co)action invariant subalgebra $B$ such that the canonical map  
	\begin{equation*}
		\can:A\otimes_B A\longrightarrow A\otimes H,
	\end{equation*}
	is bijective. These extensions are thought of as noncommutative affine principal bundles \cite{schneider1990principal}, the group's action is dualized by the (co)action of $H$ on $A$, and the freeness and transitivity on the fibers of the group action are encoded in the requirement that the canonical map is bijective.
	
	In this setting, the Chern-Weil homomorphism is built using a strong connection associated with the Hopf-Galois extension \cite{hajac1996strong,dkabrowski2001strong}, and it takes value in the cyclic homology of the base algebra $HC_*(B)$. For commutative algebras, the cyclic homology is computed by the de Rham cohomology \cite{connes1994noncommutative}; for this reason, cyclic homology is thought of as noncommutative de Rham cohomology. There are several other ways to generalize the Chern-Weil theory in the literature. One can look at the introduction of \cite{hajac2021cyclic} for an exhaustive summary.
	
	In this paper, we study how the Chern-Weil theory is affected by $2$-cocycle deformations of Hopf-Galois extensions that have been developed in a categorica way in \cite{Aschieri_2016}. This type of deformation is dual to Drinfeld twists \cite{majid_1995}, and it is a powerful tool since it gives a method to get noncommutative extensions starting from commutative ones. Notable examples are the $\mathcal{O}(SU(2))$-extension $\mathcal{O}(S_{\theta}^4)\subseteq \mathcal{O}(S_{\theta}^7)$ first studied in \cite{landi2005principal}, and $SO_{\theta}(2n,\bR)$-extension over $\mathcal{O}(S^{2n}_{\theta})$.
	
	The outline of the work is the following: in the first section, we recall the basic algebraic notions of comodule algebras and the theory of principal comodule algebras. Moreover, we review the $2$-cocycle deformation theory, giving an explicit formula for the strong connections \eqref{eq:strong_conn_def_right}, \eqref{eq:strong_conn_left} and \eqref{eq:strong_conn_both_deformed}. In the second section, we prove the naturality of characteristic classes defined by the Chern-Weil homomorphism for principal comodule algebras in Proposition \ref{prop:pushforward_Chern-Weil}, generalizing the classical result. Recalling the latter, if we denote by $\cB un_G(X)$ the category of isomorphism classes of principal $G$-bundles the manifold $X$, then the Chern-Weil homomorphism defines a map from this category to the cohomology category $H^{even}_{dR}(X)$ for any fixed element $p\in\bK[\fg]^G$, i.e. $c_X:\cB un_G(X)\longrightarrow H^{even}_{dR}(X),\quad [P]\longmapsto [p(\Omega_{\nabla_{P}})]$. For any principal $G$-bundle morphism $(F,f)$\footnote{Recall that the two maps satisfy $f\circ\pi'=\pi\circ F$, where $\pi':P'\longrightarrow X'$ and $\pi:P\longrightarrow X$ are the projections of the $G$-bundles. Therefore, the map on the base spaces $f$ is uniquely determined by the map on the principal spaces $F$.}, $F:P'\longrightarrow P$ and $f:X'\longrightarrow X$, one has $\cB un_G(f)=f^*$ is the pullback and maps the class of $P$ in $\cB un_G(X)$ in that of $f^*P$, and the following diagram is commutative
	\begin{equation*}
		\begin{tikzcd}
			\cB un_G(X) \arrow[rr, "\cB un_G(f)"] \arrow[dd, "c_X"'] & & \cB un_G(X') \arrow[dd, "c_{X'}"]\\
			& & \\
			H^{even}_{dR}(X) \arrow[rr, "H^{even}_{dR}(f)"] & &  H^{even}_{dR}(X') 
		\end{tikzcd}
	\end{equation*}
	The proof of this fact is tantamount to proving that the Chern-Weil homomorphism $\mathrm{Cwh}_P$ associated to $P\longrightarrow X$ satisfies $\mathrm{Cwh_{f^*P}}=f^*\circ\mathrm{Cwh}_P$ for any principal $G$-bundle morphism $(F,f)$ as above. Finally, in the last section, we study the Chern-Weil homomorphism of \cite{hajac2021cyclic} under $2$-cocycle deformations. The main results of this section are \ref{prop:Chern-Weil_deform_right} and \ref{prop:deforma_Chern-Weil_external}. The first shows that the map is unchanged under a deformation coming from a $2$-cocycle of the structure Hopf algebra. Geometrically, in this situation, only the fibers are deformed, and the base manifold remains the same. In the second, we show that the Chern-Weil homomorphism is deformed and we give an explicit formula for it. Here, both the fibers and the base manifold are deformed in the geometric picture. We then analyze what happens when both deformations are performed and we also discuss the naturality of characteristic classes in the context of $2$-cocycle deformations \ref{prop:deformed_pullback}.

	\section{Preliminaries}
	In this section we recall the basic notions of comodule, comodule algebra, and Hopf-Galois extension with a strong connection, mostly to fix our notation. In the last two subsections we find explicit formulas for a strong connecion of a $2$-cocycle deformed principal comodule algebra. 
	
	We work over a fixed field $\mathbb{K}$ and put $\otimes:=\otimes_{\mathbb{K}}$. Algebras and Hopf algebras are assumed to be unital and counital, as well as associative and coassociative. We denote the unit with $1_{-}$ and counit $\epsilon_{-}$, for the produt we use juxtaposition and for the coproduct Sweedler notation $\Delta(h):=h_{(1)}\otimes h_{(2)}$ where $h$ is any element in a Hopf algebra $H$. The antipode is denoted by $S$ and in what follows is assumed to be invertible.
	
	\subsection{Principal comodule algebras}  
	A \textit{right $H$-comodule} is a vector space $V$ with map $\rho_V:V\longrightarrow V\otimes H$ satisfying $(\rho_V\otimes\mathrm{id}_H)\circ\rho_V=(\mathrm{id}_V\otimes\Delta)\circ\rho_V$ and $(\mathrm{id}_V\otimes\epsilon)\circ\rho_V=\mathrm{id}_V$ that is called \textit{coaction}. We use the notation $\rho_V(v):=v_{(0)}\otimes v_{(1)}$. a \textit{morphism} of right $H$-comodules is a linear map $f:V\longrightarrow W$ such that $\rho_W\circ f=(f\otimes\id_H)\circ\rho_V$. We denote by $\fM^H$ the category of right $H$-comodules. The latter is a \textit{monoidal} category, the tensor product $V\otimes W$ of two $H$-comodules is a $H$-comodule with the \textit{diagonal coaction} $\rho^{\otimes}(v\otimes w)=v_{(0)}\otimes w_{(0)}\otimes v_{(1)}w_{(1)}$. The definitions for the category of \textit{left} $H$-comodules $\tensor[^H]{\fM}{}$ are similar. Using the bijectivity of the antipode $S$ of $H$, any $V\in\fM^H$ is also a left $H$-comodule, being the map $\lambda_{V}(v)=S^{-1}(v_{(1)})\otimes v_{(0)}$ a left coaction. If $V$ is a right $H$-comodule with coaction $\rho$ and $W$ a left $H$-comodule with coaction $\lambda$, we can form the so-called \textit{cotensor product}
	\beq
	\label{eq:cotensor_prod}
	V\square^H W:=\{v\otimes w\in V\otimes W|\rho(v)\otimes w=v\otimes\lambda(w)\}. 
	\eeq
	This operation is a bi-functor $\fM^H\times\tensor[^H]{\fM}{}\longrightarrow\mathrm{Vect}$ into the category of vector spaces. 
	\begin{expl}
		\label{expl:adjoint_coaction}
		Any Hopf algebra $H$ is a right $H$-comodule if endowed with the product $\rho:=\Delta$. Moreover, one can define the \textit{adjoint coaction}
		\beq
		\label{eq:adjoint_coaction}
		\mathrm{Ad}:H\longrightarrow H\otimes H,\quad h\longmapsto h_{(2)}\otimes S(h_{(1)})h_{(3)}.
		\eeq
		$H$ equipped with this right coaction is denote by $\underline{H}$.    
	\end{expl}
	
	An algebra $A$ is a \textit{right $H$-comodule algebra} if it is a right $H$-comodule and the coaction $\rho_A$ is an algebra morphism. The sub-algebra of \textit{coaction invariant elements} is denoted by $B:=\{b\in A|\rho_A(b)=b\otimes 1_H\}$. A morphism of right $H$-comodule algebras is a right $H$-comodule morphism that is also an algebra morphism. The category of right $H$-comodule algebras is denoted by $\cA^H$.
	
	Given a right $H$-comodule algebra $A$, we say that $B\subseteq A$ is a \textit{$H$-Hopf-Galois extension} if the canonical map
	\begin{equation*}
		\mathrm{can}:A\otimes_B A\longrightarrow A\otimes H, \quad a\otimes_B\tilde{a}\longrightarrow a\tilde{a}_{(0)}\otimes \tilde{a}_{(1)} 
	\end{equation*}    
	is bijective, where $A\otimes_BA$ is the balanced tensor product. A Hopf-Galois extension such that $A$ is a faithfully flat left $B$-module is called a \textit{principal $H$-comodule algebra}. Faithfully flatness of $A$ as left a $B$-module means that the functor $-\otimes_BA$ from the category of right $B$-modules to one of the vector spaces reflects and preserves exact sequences.
	
	A \textit{strong connection} \cite{hajac1996strong,dkabrowski2001strong} on a $H$-Hopf-Galois extension is a linear map $\ell:H\longrightarrow A\otimes A$ such that the following equations hold
	\begin{align}
		\label{eq:strong_conn1}
		(\mathrm{id}_A\otimes\rho)\circ \ell&=(\ell\otimes\mathrm{id}_H)\circ\Delta,\\
		\label{eq:strong_conn2}
		(\lambda\otimes\mathrm{id}_A)\circ \ell&=(\mathrm{id}_H\otimes \ell)\circ\Delta,\\
		\label{eq:strong_conn3}
		\pi_B\circ \ell&=\tau.
	\end{align}
	We have denoted by $\tau$ the translation map \cite{Schneider1990RepresentationTO,Brzezi_ski_1996} of the extension $B\subseteq A$ which is the restrictions
	\begin{equation*}
		\tau:=\can^{-1}|_H:H\longrightarrow A\otimes_B A.
	\end{equation*}
	By definition we then have $\can(\tau(h))=1_A\otimes h,\quad \forall h\in H$.
	
	An extension $B\subseteq A$ admits a strong connection $\ell:H\longrightarrow A\otimes A$ if and only if is a principal $H$-comodule algebra  \cite{schauenburg2005generalized,dkabrowski2001strong}. For the latter we use the notation $\ell(h):=h^{\1}\otimes h^{\2}$, so that the identities \eqref{eq:strong_conn1}-\eqref{eq:strong_conn3} read
	\begin{align}
		\label{eq:strong_conn_Sw1}
		h^{\1}\otimes \tensor{h}{^{\2}_{(0)}}\otimes\tensor{h}{^{\2}_{(1)}}&=\tensor{h}{_{(1)}^{\1}}\otimes \tensor{h}{_{(1)}^{\2}}\otimes h_{(2)},\\
		\label{eq:strong_conn_Sw2}
		\tensor{h}{^{\1}_{(0)}}\otimes h^{\2}\otimes\tensor{h}{^{\1}_{(1)}}&=\tensor{h}{_{(2)}^{\1}}\otimes \tensor{h}{_{(2)}^{\2}}\otimes S(h_{(1)}),\\
		\label{eq:strong_conn_Sw3}
		(\can\circ\pi_B\circ\ell)(h)&=1_A\otimes h.
	\end{align} 
	If we apply the map $\id_A\otimes\epsilon$ on both sides of \eqref{eq:strong_conn_Sw3}, we get the equation
	\begin{equation*}
		h^{\1}h^{\2}=\epsilon(h)1_A,\quad \forall h\in H.
	\end{equation*}
	Here $\pi_B$ is the quotient map from $A\otimes A$ onto $A\otimes_B A$, thus \eqref{eq:strong_conn_Sw3} tells us that a strong connection is a lift of the translation map to the whole $A\otimes A$.
	
	There is a special class of principal comodule algebras, called \textit{cleft} extensions, that are characterized by having a unital, convolution invertible map $\phi: H\longrightarrow A$ called \textit{cleaving map}. This condition is tantamount to the \textit{normal basis} property of the extension, i.e. $A\simeq B\otimes H$ as left $B$-module and right $H$-comodule \cite{Montgomery1993HopfAA}. The strong connection for cleft extensions is given by
	\begin{equation}
		\label{eq:strong_conn_cleft}
		\ell(h)=\phi^{-1}(h_{(1)})\otimes\phi(h_{(2)}),\quad h\in H.
	\end{equation}
	
	\subsection{Deformations via $2$-cocycles}
	We briefly review the concept of $2$-cocycle deformation of Hopf algebras, which is dual to the Drinfeld twist \cite[Chapter 2]{majid_1995}, and the related deformation theory of comodule algebras.
	
	Recall that for every Hopf algebra $H$ the tensor product $H\otimes H$ is a bialgebra with component-wise multiplication and unit, and comultiplication $\Delta^{\otimes}(h\otimes k)=h_{(1)}\otimes k_{(1)}\otimes h_{(2)}\otimes k_{(2)}$ and counit $\epsilon^{\otimes}(h\otimes k)=\epsilon(h)\epsilon(k)$. A \textbf{$2$-cocycle} is a map $\gamma:H\otimes H\longrightarrow\bK$ that is convolution invertible and satisfies the conditions
	\begin{align}
		\label{eq:cocycle_1}
		\gamma(h_{(1)}\otimes k_{(1)})\gamma(h_{(2)}k_{(2)}\otimes l)&=\gamma(k_{(1)}\otimes l_{(1)})\gamma(h\otimes k_{(2)}l_{(2)}),\\
		\label{eq:cocycle_2}
		\gamma(h\otimes 1_H)&=\gamma(1_H\otimes h)=\epsilon(h),
	\end{align}
	for all $h,k,l\in H$. For writing convenience, we use the notation $\gamma(h,k):=\gamma(h\otimes k)$ for the rest of the paper. The following is the dual version of Theorem 2.3.4 in \cite{majid_1995}.
	
	%\begin{expl}
	%    For the Hopf algebra $O(G)$ of regular functions on an algebraic group, the $2$-cocycles are indeed the group $2$-cocycles $Z^2(G)$ i.e. the functions $f:O(G)\otimes O(G)\simeq O(G\times G)\longrightarrow\bK$ such that
	%    \begin{equation*}
	%        f(g,h)f(gh,k)=f(h,k)f(g,hk),\quad f(g,e_G)=f(e_G,g)=1,
	%    \end{equation*}
	%    for all $g, h, k\in G$. Invertibility in this case means that $f$ is nowhere zero on $G\times G$. 
	%\end{expl}
	
	\begin{prop}
		\label{prop:deformation_Hopf_alg}
		Given a $2$-cocycle $\gamma:H\otimes H\longrightarrow\bK$ of a Hopf algebra $H$, the equation
		\begin{equation*}
			h\cdot_{\gamma}k:=\gamma(h_{(1)}, k_{(1)})h_{(2)}k_{(2)}\gamma^{-1}(h_{(3)}, k_{(3)}),\quad h,k\in H
		\end{equation*}
		defines a new associative multiplication on $H$, the resulting algebra is denoted by $H_{\gamma}$. The latter is a Hopf algebra with the comultiplication and counit inherited from $H$ and antipode
		\begin{equation*}
			S_{\gamma}(h)=u_{\gamma}(h_{(1)})S(h_{(2)})u^{-1}_{\gamma}(h_{(3)}),\quad h\in H_{\gamma},
		\end{equation*}
		where $u_{\gamma}:H\longrightarrow\bK$ is the convolution invertible map given by $u_{\gamma}(h)=\gamma(h_{(1)}, S(h_{(2)}))$. Moreover, if $S$ is bijective so is $S_{\gamma}$, its inverse given by
		\begin{equation*}
			S^{-1}_{\gamma}(h)=v_{\gamma}(h_{(1)})S^{-1}(h_{(2)})v^{-1}_{\gamma}(h_{(3)}),\quad h\in H_{\gamma}
		\end{equation*}
		where $v_{\gamma}:H\longrightarrow\mathbb{K}$ mapping $h\longmapsto\gamma(h_{(2)},S^{-1}(h_{(1)}))$ is convolution invertible.
	\end{prop}

	Now let $V$ be a right $H$-comodule. Since the deformation of $H_{\gamma}$ of $H$ by a $2$-cocycle does not involve the coalgebra structure, we have that $V$ is a right $H_{\gamma}$-comodule. The coaction is unchanged. Taking another right $H$-comodule $W$ with corresponding right $H_{\gamma}$-comodule, we have the tensor product $V_{\gamma}\otimes^{\gamma}W_{\gamma}$ is the vector space $V\otimes W$ endowed with the diagonal right $H_{\gamma}$-coaction
	\beq
	\label{eq:diagonal_right_coaction_gamma}
	\rho^{\otimes^{\gamma}}:V_{\gamma}\otimes^{\gamma}W_{\gamma}\longrightarrow V_{\gamma}\otimes^{\gamma}W_{\gamma}\otimes H_{\gamma},\quad v\otimes^{\gamma}w\longmapsto v_{(0)}\otimes^{\gamma}w_{(0)}\otimes v_{(1)}\cdot_{\gamma}w_{(1)}.
	\eeq
	So we have just seen that the $2$-cocycle $\gamma$ induces a functor from $\fM^H$, the category of right $H$-comodules, to $\fM^{H_{\gamma}}$ the category of right $H_{\gamma}$-comodules, and that the latter is also monoidal. Moreover, we have the following \cite[Theorem 2.19]{Aschieri_2016}
	\begin{thm}[]
		\label{thm:comodule_equiv_2-co}
		The functor $\Gamma:\fM^H\longrightarrow \fM^{H_{\gamma}}$ mapping $V\longmapsto V_{\gamma}$  is an equivalence of (monoidal) categories, the linear map
		\begin{equation*}
			\alpha_{V,W}:V_{\gamma}\otimes^{\gamma}W_{\gamma}\longrightarrow (V\otimes W)_{\gamma},\quad v\otimes^{\gamma}w\longmapsto v_{(0)}\otimes w_{(0)}\gamma^{-1}(v_{(1)}, w_{(1)})
		\end{equation*}
		is a right $H_{\gamma}$-comodule isomorphism with inverse
		\begin{equation*}
			\alpha^{-1}_{V,W}:(V\otimes W)_{\gamma}\longrightarrow V_{\gamma}\otimes^{\gamma}W_{\gamma} ,\quad v\otimes w\longmapsto v_{(0)}\otimes^{\gamma} w_{(0)}\gamma(v_{(1)}, w_{(1)}).
		\end{equation*}
	\end{thm}
	
	If $A$ is a right $H$-comodule algebra, we can deform the product of $A$ using th $2$-cocycle $\gamma$ by defining
	\beq
	\label{eq:deformed_prod_right_algebra}
	a\cdot_{\gamma}\tilde{a}:=a_{(0)}\tilde{a}_{(0)}\gamma^{-1}(a_{(1)}, \tilde{a}_{(1)}).
	\eeq  
	We denote by $A_{\gamma}$ the resulting algebra. It is straightforward, using the properties of $2$-cocycles, that the right coaction $\rho$ of $H_{\gamma}$ on $A_{\gamma}$ is still an algebra morphism with respect to the new product $\rho(a\cdot_{\gamma}\tilde{a})=\rho(a)\cdot_{\gamma}\rho(\tilde{a})$, $\forall a,\tilde{a}\in A_{\gamma}$ making $A_{\gamma}$ a right $H_{\gamma}$-comodule algebra. Thus, also for comodule algebras we a functor $\cA^H\longrightarrow \cA^{H_{\gamma}}$.

	For any left $H$-comodules a similar result holds. To distinguish the two cases we denote the $2$-cocycle by $\sigma:H\otimes H\longrightarrow\bK$. In this case, the functor realizing the equivalence between left $H$-comodules $(\tensor[^H]{\mathfrak{M}}{},\otimes)$ and left $H_{\sigma}$-comodules $(\tensor[^{H_{\sigma}}]{\mathfrak{M}}{},\tensor[^{\sigma}]{\otimes}{})$ is denoted by $\Sigma:V\longmapsto \tensor[_{\sigma}]{V}{}$. It is a monoidal functor with isomorphism given by
	\beq
	\label{eq:deformed_tensor_iso_left}
	\phi_{V,W}:\tensor[_{\sigma}]{V}{}\tensor[^{\sigma}]{\otimes}{}\tensor[_{\sigma}]{W}{}\longrightarrow\tensor[_{\sigma}]{(V\otimes W)}{},\quad v\tensor[^{\sigma}]{\otimes}{}w\longmapsto\sigma({v_{(-1)}, w_{(-1)}})v_{(0)}\otimes w_{(0)},
	\eeq
	which inverse is given by
	\beq
	\label{eq:deformed_tensor_iso_left_inv}
	\phi^{-1}_{V,W}:\tensor[_{\sigma}]{(V\otimes W)}{}\longrightarrow\tensor[_{\sigma}]{V}{}\tensor[^{\sigma}]{\otimes}{}\tensor[_{\sigma}]{W}{},\quad v\otimes w\longmapsto\sigma^{-1}({v_{(-1)}, w_{(-1)}})v_{(0)}\tensor[^{\sigma}]{\otimes}{} w_{(0)}.
	\eeq
	The product in a left $H$-comodule algebra is deformed in the following way
	\beq
	\label{eq:deform_prod_left_algebra}
	a\bullet_{\sigma}\tilde{a}:=\sigma(a_{(-1)}, \tilde{a}_{(-1)})a_{(0)}\tilde{a}_{(0)}.
	\eeq
	The resulting left $H_{\sigma}$-comodule algebra is denoted by $\tensor[_{\sigma}]{A}{}$.
	
	\subsection{External symmetries}
	Let $H$ and $K$ be two Hopf algebras. A $(K,H)$-bicomodule is a vector space $V$ endowed with a left $K$-coaction $\lambda_V$ and a right $H$-coaction $\rho_V$ such that they commute
	\beq
	\label{eq:bicomodule_eq}
	(\lambda_V\otimes\mathrm{id}_H)\circ\rho_V=(\mathrm{id}_K\otimes\rho_V)\circ\lambda_V.
	\eeq 
	In Sweedler notation we have $\rho_V(v)=v_{(0)}\otimes v_{(1)}$, $\lambda_V(v)=v_{(-1)}\otimes v_{(0)}$ then the above equation reads
	\begin{equation*}
		v_{(0)(-1)}\otimes v_{(0)(0)}\otimes v_{(1)}=v_{(-1)}\otimes v_{(0)(0)}\otimes v_{(0)(1)},\quad v\in V.
	\end{equation*}
	A $(K,H)$-bicomodule algebra is an algebra $A$ with commuting $(K,H)$-coactions as above that are also algebra morphism. In general one has that $A^{coH}$ is different from $\tensor[^{coK}]{A}{}$. When we consider a $H$-Hopf-Galois extension $B\subseteq A$ and a Hopf algebra $K$ such that $A$ is a $(K,H)$-comodule algebra, we refer to $K$ as an \textit{external symmetry} as in \cite{Aschieri_2016}.   
	
	\subsection{Deforming principal comodule algebras}
	We now characterize the deformations of Hopf-Galois extensions and principal comodule algebras. Regarding the former, we report the result of \cite{Aschieri_2016}, for the latter we give explicit formulas for a strong connection $\ell$ that were given implicitly in the reference.
	
	We first consider a right $H$-Hopf-Galois extension $B\subseteq A$ and a $2$-cocycle $\gamma$ of $H$ that deforms $H$ into $H_{\gamma}$ and $A$ into $A_{\gamma}$. The subalgebra $B$ is unchanged, in fact for any pair $b,b'\in B$ one has $b\cdot_{\gamma}b'=bb'$.
	
	\begin{thm}[]
		\label{thm:deformed_Hopf_alg}
		The algebra extension $B\subseteq A$ is $H$-Hopf-Galois if and only if $B\subseteq A_{\gamma}$ is $H_{\gamma}$-Hopf-Galois.
	\end{thm} 
	This result \cite[Theorem 3.6]{Aschieri_2016} is a consequence of the commutativity of the following diagram
	\beq
	\label{eq:diagrm_Hopf-Galois}
	\begin{tikzcd}
		A_{\gamma}\otimes_BA_{\gamma} \arrow[rr, "\mathrm{can}_{\gamma}"] \arrow[dd, "\simeq"'] & & A_{\gamma}\otimes \underline{H_{\gamma}} \arrow[dd, "\simeq"]\\
		& & \\
		(A\otimes_BA)_{\gamma} \arrow[rr, "\Gamma(\mathrm{can})"'] & & (A\otimes \underline{H})_{\gamma},
	\end{tikzcd}
	\eeq
	where $\underline{H}$ is $H$ endowed with the right adjoint coaction of the Example \ref{expl:adjoint_coaction}, the left vertical arrow is the isomorphism of \ref{thm:comodule_equiv_2-co} for $A\otimes_B A$ while the right one is always given by \ref{thm:comodule_equiv_2-co} composed with the isomorphism of Theorem 3.4 in \cite{Aschieri_2016}. Here $H$ and $H_{\gamma}$ are respectively right $H$ and $H_{\gamma}$-comodule endowed with the adjoint coaction \eqref{eq:adjoint_coaction}. The bijectivity of $\can_{\gamma}$ is tantamount to the bijectivity of $\Gamma(\can)$. 
	
	The same result holds for principal comodule algebras, here we prove it by showing the explicit formula for a strong connection
	\begin{thm}
		\label{thm:deformed_strong_conn_right}
		The algebra extension $B\subseteq A$ is a principal $H$-comodule algebra if and only if $B\subseteq A_{\gamma}$ is a principal $H_{\gamma}$-comodule algebra. 
	\end{thm}
	\begin{proof}
		To write down the right diagram, we need to introduce the $H_{\gamma}$-comodule isomorphism \cite[Theorem 3.4]{Aschieri_2016}
		\begin{equation*}
			\ff:\underline{H_{\gamma}}\longrightarrow\underline{H}_{\gamma},\quad h\longmapsto h_{(3)}u_{\gamma}(h_{(1)})\gamma^{-1}(S(h_{(2)}),h_{(4)})
		\end{equation*}
		whose inverse is given by
		\begin{equation*}
			\ff^{-1}:\underline{H}_{\gamma}\longrightarrow\underline{H_{\gamma}},\quad h\longmapsto h_{(3)}u^{-1}_{\gamma}(h_{(2)})\gamma(S(h_{(1)}),h_{(4)}).
		\end{equation*}
		We recall that $\underline{H}$ indicates $H$ endowed with the right adjoint coaction \eqref{eq:adjoint_coaction}. At this point, defines $\ell_\gamma$ the map making the following diagram commute
		\begin{equation*}
			\begin{tikzcd}
				\underline{H}_{\gamma} \arrow[rr, "\Gamma(\ell)"] \arrow[dd, "\ff^{-1}"'] & & (A\otimes A)_{\gamma} \arrow[dd, "\alpha^{-1}_{A,A}"]\\
				& &\\
				\underline{H_{\gamma}} \arrow[rr, "\ell_{\gamma}"] & & A_{\gamma}\otimes^{\gamma}A_{\gamma}
			\end{tikzcd}
		\end{equation*}
		Explicitly, we have
		\begin{equation*}
			\ell_{\gamma}:=\alpha^{-1}_{A,A}\circ\Gamma(l)\circ\ff:\underline{H_{\gamma}}\longrightarrow A_{\gamma}\otimes^{\gamma}A_{\gamma}
		\end{equation*}
		mapping $h\longmapsto \tensor[]{h}{_{(2)}^{\1}}\otimes^{\gamma}\tensor[]{h}{_{(2)}^{\2}}u_{\gamma}(h_{(1)})$. We now prove that it is a strong connection for the extension $B\subseteq A_{\gamma}$. For this, we need to check the defining equations for any $h\in H_{\gamma}$
		\begin{equation*}
			\begin{split}
				[(\mathrm{id}_A\otimes\rho^{\gamma})\circ \ell_{\gamma}](h)&=\tensor{h}{_{(2)}^{\1}}\otimes^{\gamma}\tensor{h}{_{(2)}^{\2}_{(0)}}\otimes\tensor{h}{_{(2)}^{\2}_{(1)}}u_{\gamma}(h_{(1)})\\
				&=\tensor{h}{_{(2)}^{\1}}\otimes^{\gamma}\tensor{h}{_{(2)}^{\2}}\otimes h_{(3)}u_{\gamma}(h_{(1)})\\
				&=[(\mathrm{id}_H\otimes \ell_{\gamma})\circ\Delta](h)
			\end{split}
		\end{equation*}
		
		\begin{equation*}
			\begin{split}
				[(\lambda^{\gamma}\otimes\mathrm{id}_A)\circ \ell_{\gamma}](h)&=S_{\gamma}^{-1}(\tensor{h}{_{(2)}^{\1}_{(1)}})\otimes\tensor{h}{_{(2)}^{\1}_{(0)}}\otimes\tensor{h}{_{(2)}^{\2}}u_{\gamma}(h_{(1)})\\
				&=S^{-1}_{\gamma}(S(h_{(2)}))\otimes\tensor{h}{_{(3)}^{\1}}\otimes^{\gamma}\tensor{h}{_{(3)}^{\1}}u_{\gamma}(h_{(1)})\\
				&=S^{-1}_{\gamma}(S_{\gamma}(h_{(3)})\otimes\tensor{h}{_{(5)}^{\1}}\otimes^{\gamma}\tensor{h}{_{(5)}^{\2}}u_{\gamma}(h_{(1)})u_{\gamma}^{-1}(h_{(2)})u_{\gamma}(h_{(4)})\\
				&=h_{(1)}\otimes\tensor{h}{_{(3)}^{\1}}\otimes^{\gamma}\tensor{h}{_{(3)}^{\2}}u_{\gamma}(h_{(2)})=\left[(\mathrm{id}_H\otimes \ell_{\gamma})\circ\Delta\right](h) 
			\end{split}
		\end{equation*}
		
		\begin{equation*}
			\begin{split}
				(\can_{\gamma}\circ\pi_B\circ\ell_{\gamma})(h)&=\tensor{h}{_{(2)}^{\1}}\cdot_{\gamma}\tensor{h}{_{(2)}^{\2}_{(0)}}\otimes\tensor{h}{_{(2)}^{\2}_{(1)}}u_{\gamma}(h_{(1)})\\
				&=\tensor{h}{_{(2)}^{\1}}\cdot\tensor{h}{_{(2)}^{\2}}\otimes h_{(3)}u_{\gamma}(h_{(1)})\\
				&=1_A\otimes h_{(5)}\gamma(h_{(1)},S(h_{(2)}))\gamma^{-1}(S(h_{(3)}),h_{(4)})\\
				&=1_A\otimes h_{(2)}(u_{\gamma}*u^{-1}_{\gamma})(h_{(1)})=1_A\otimes h
			\end{split}
		\end{equation*}
		We used \eqref{eq:strong_conn_Sw1}, \eqref{eq:strong_conn_Sw2}, the definition of product $\cdot_{\gamma}$ \eqref{eq:deformed_prod_right_algebra}, and \eqref{eq:strong_conn_Sw3}. Being $\ff$ and $\alpha_{A,A}$ isomorphism, the inverse implication is also true.
	\end{proof}
	Thus for the deformed principal $H_{\gamma}$-comodule algebra $B\subseteq A_{\gamma}$ a strong connection is given by the formula
	\beq
	\label{eq:strong_conn_def_right}
	\ell_{\gamma}(h)=\tensor{h}{_{(2)}^{\1}}\otimes^{\gamma}\tensor{h}{_{(2)}^{\2}}u_{\gamma}(h_{(1)}),\quad h\in H_{\gamma}.
	\eeq
	
	If $B\subseteq A$ is an $H$-cleft extension with cleaving map $\phi$, then the deformed extension is cleft, and in this case too, the opposite implication is true. For a proof of this fact, we refer to \cite[Corollary 3.7]{Aschieri_2016} where the categorical setting of deformations is considered, and also to \cite[Theorem 5.3]{montgomery2005krull} where an explicit formula for the deformed inverse cleaving map is given, namely
	\begin{equation}
		\label{eq:cleaving_inv_deformed}
		\phi_{\gamma}^{-1}(h)=\phi^{-1}(h_{(2)})u_{\gamma}(h_{(1)}),\quad h\in H.
	\end{equation}
	Thus, using the identy \eqref{eq:strong_conn_cleft}, we have
	\begin{equation}
		\label{eq:strong_conn_cleft_deformed_right}
		\ell_{\gamma}(h)=\phi^{-1}(h_{(2)})\otimes \phi(h_{(3)})u_{\gamma}(h_{(1)}),\quad h\in H,
	\end{equation}
	which agrees with the result of \ref{thm:deformed_strong_conn_right}.

	Now let $A$ be a $(K,H)$-bicomodule algebra and put $B:=A^{coH}$. If a $2$-cocycle $\sigma$ of $K$ is given, we denote by $\Sigma:\tensor[^K]{\mathfrak{M}}{}\longrightarrow\tensor[^{K_{\sigma}}]{\mathfrak{M}}{}$ the corresponding monoidal functor, so $\Sigma(V)=\tensor[_{\sigma}]{V}{}$ for any left $K$-comodule. The Hopf algebra $H$ is a left $K$-comodule algebra with the trivial coaction $h\longmapsto 1_K\otimes h$, so one has $\tensor[_{\sigma}]{H}{}\simeq H$. 
	
	\begin{thm}[]
		\label{thm:deformed_ext_symm}
		The algebra extension $B\subseteq A$ is $H$-Hopf-Galois if and only if $\tensor[_{\sigma}]{B}{}\subseteq\tensor[_{\sigma}]{A}{}$ is $H$-Hopf-Galois.
	\end{thm}
	As for the case where the Hopf algebra $H$ is deformed, the proof of this result \cite[Theorem 3.15 and Corollary 3.16]{Aschieri_2016} follows from the commutativity of a diagram, namely
	\beq
	\begin{tikzcd}
		\tensor[_{\sigma}]{A}{}\tensor[^{\sigma}]{\otimes}{_{\tensor[_{\sigma}]{B}{}}}\tensor[_{\sigma}]{A}{} \arrow[rr, "\mathrm{can}_{\sigma}"] \arrow[dd, "\simeq"'] & & \tensor[_{\sigma}]{A}{}\tensor[^{\sigma}]{\otimes}{}H \arrow[dd, "\simeq"]\\
		& & \\
		\tensor[_{\sigma}]{(A\otimes_BA)}{} \arrow[rr, "\Sigma(\mathrm{can})"] & & \tensor[_{\sigma}]{(A\otimes H)}{}, 
	\end{tikzcd}
	\eeq
	where the vertical arrows are the corresponding isomorphism \eqref{eq:deformed_tensor_iso_left}. 
	
	Again the same result holds for principal comodule algebras
	\begin{thm}
		\label{thm:strong_conn_left}
		The algebra extension $B\subseteq A$ with external symmetry $K$ is a principal $H$-comodule algebra if and only if $\tensor[_{\sigma}]{B}{}\subseteq\tensor[_{\sigma}]{A}{}$ is a principal $H$-comodule algebra with external symmetry $\tensor[_{\sigma}]{K}{ }$. 
	\end{thm}
	\begin{proof}
		Recall that the inverse isomorphism of left the $K$-comodule \eqref{eq:deformed_tensor_iso_left} $A\otimes A$ is given by
		\begin{equation*}
			\phi^{-1}_{A,A}(a\otimes \tilde{a})=\sigma^{-1}\left(a_{(-1)}, \tilde{a}_{(-1)}\right)a_{(0)}\tensor[^{\sigma}]{\otimes}{} \tilde{a}_{(0)}.
		\end{equation*}
		We define the map $\tensor[_{\sigma}]{\ell}{}:=\phi^{-1}_{A,A}\circ\Sigma(l):H\longrightarrow\tensor[_{\sigma}]{A}{}\tensor[^{\sigma}]{\otimes}{}\tensor[_{\sigma}]{A}{}$ sending
		\begin{equation*}
			h\longmapsto \sigma^{-1}\left(\tensor{h}{^{\1}_{(-1)}},\tensor{h}{^{\2}_{(-1)}}\right)\tensor{h}{^{\1}_{(0)}}\tensor[^{\sigma}]{\otimes}{}\tensor{h}{^{\2}_{(0)}},
		\end{equation*}
		which makes the following diagram commute
		\begin{equation*}
			\begin{tikzcd}
				H \arrow[rr, "_{\sigma}\ell"] \arrow[rrdd, "\Sigma(\ell)"'] & & \tensor[_{\sigma}]{A}{}\tensor[^{\sigma}]{\otimes}{}\tensor[_{\sigma}]{A}{} \arrow[dd, "\phi_{A,A}"]\\
				& &\\
				&	& \tensor[_{\sigma}]{(A\otimes A)}{}
			\end{tikzcd}
		\end{equation*}
		Assume that $l$ is a strong connection of $B\subseteq A$, then for $\tensor[_{\sigma}]{\ell}{}$ we find that for any $h\in H$
		\begin{equation*}
			\begin{split}
				\left[(\mathrm{id}_{\tensor[_{\sigma}]{A}{}}\otimes\rho)\circ\tensor[_{\sigma}]{\ell}{}\right](h)&=\sigma^{-1}\left(\tensor{h}{^{\1}_{(-1)}},\tensor{h}{^{\2}_{(-1)}}\right)\tensor{h}{^{\1}_{(0)}}\tensor[^{\sigma}]{\otimes}{}\tensor{h}{^{\2}_{(0)(0)}}\otimes\tensor{h}{^{\2}_{(0)(1)}}\\
				&=\sigma^{-1}\left(\tensor{h}{^{\1}_{(0)}},\tensor{h}{^{\2}_{(0)(-1)}}\right)\tensor{h}{^{\1}_{(0)}}\tensor[^{\sigma}]{\otimes}{}\tensor{h}{^{\2}_{(0)(0)}}\otimes\tensor{h}{^{\2}_{(1)}}\\
				&=\sigma^{-1}\left(\tensor{h}{_{(1)}^{\1}_{(0)}},\tensor{h}{_{(1)}^{\2}_{(-1)}}\right)\tensor{h}{_{(1)}^{\1}_{(0)}}\otimes\tensor{h}{_{(1)}^{\2}_{(0)}}\tensor[^{\sigma}]{\otimes}{}\tensor{h}{_{(2)}^{\2}}\\
				&=\left[(\tensor[_{\sigma}]{\ell}{}\otimes\mathrm{id}_H)\circ\Delta\right](h),
			\end{split}
		\end{equation*}
		
		\begin{equation*}
			\begin{split}
				\left[(\lambda\otimes\mathrm{id}_{\tensor[_{\sigma}]{A}{}})\circ\tensor[_{\sigma}]{\ell}{}\right](h)&=\sigma^{-1}\left(\tensor{h}{^{\1}_{(-1)}},\tensor{h}{^{\2}_{(-1)}}\right)S^{-1}\left(\tensor{h}{^{\1}_{(0)(1)}}\right)\otimes\tensor{h}{^{\1}_{(0)(0)}}\tensor[^{\sigma}]{\otimes}{}\tensor{h}{^{\2}_{(0)}}\\
				&=\sigma^{-1}\left(\tensor{h}{^{\1}_{(0)(-1)}},\tensor{h}{^{\2}_{(-1)}}\right)S^{-1}\left(\tensor{h}{^{\1}_{(1)}}\right)\otimes\tensor{h}{^{\1}_{(0)(0)}}\tensor[^{\sigma}]{\otimes}{}\tensor{h}{^{\2}_{(0)}}\\
				&=S^{-1}\left(S(h_{(1)})\right)\otimes\sigma^{-1}\left(\tensor{h}{_{(2)}^{\1}_{(-1)}},\tensor{h}{_{(2)}^{\2}_{(-1)}}\right)\tensor{h}{_{(2)}^{\1}_{(0)}}\tensor[^{\sigma}]{\otimes}{}\tensor{h}{_{(2)}^{\2}_{(0)}}\\
				&=\left[(\mathrm{id}_H\otimes\tensor[_{\sigma}]{\ell}{})\circ\Delta\right](h),\\
			\end{split}
		\end{equation*}
		
		\begin{equation*}
			\begin{split}
				(\tensor[_{\sigma}]{\can}{}\circ\pi_{\tensor[_{\sigma}]{B}{}}\circ\tensor[_{\sigma}]{\ell}{})(h)&=\sigma^{-1}\left(\tensor{h}{^{\1}_{(-1)}},\tensor{h}{^{\2}_{(-1)}}\right)\tensor{h}{^{\1}_{(0)}}\bullet_{\sigma}\tensor{h}{^{\2}_{(0)}_{(0)}}\otimes\tensor{h}{^{\2}_{(1)}}\\
				&=\sigma\left(\tensor{h}{_{(1)}^{\1}_{(0)(-1)}},\tensor{h}{_{(1)}^{\2}_{(0)(-1)}}\right)\sigma^{-1}\left(\tensor{h}{_{(1)}^{\1}_{(-1)}},\tensor{h}{_{(1)}^{\2}_{(-1)}}\right)\\
				&\tensor{h}{_{(1)}^{\1}_{(0)}}\tensor{h}{_{(1)}^{\2}_{(0)}}\otimes h_{(2)}\\
				&=(\sigma*\sigma^{-1})\left(\tensor{h}{_{(1)}^{\1}_{(-1)}},\tensor{h}{_{(1)}^{\2}_{(-1)}}\right)\tensor{h}{_{(1)}^{\1}_{(0)}}\tensor{h}{_{(1)}^{\2}_{(0)}}\otimes h_{(2)}\\
				&=\tensor{h}{_{(1)}^{\1}}\tensor{h}{_{(1)}^{\2}}\otimes h_{(2)}=1_A\otimes h,
			\end{split}
		\end{equation*}
		proving that $\tensor[_{\sigma}]{\ell}{}$ is a strong connection for $\tensor[_{\sigma}]{B}{}\subseteq\tensor[_{\sigma}]{A}{}$. We used the equations \eqref{eq:strong_conn_Sw1}-\eqref{eq:strong_conn_Sw3}, that $\tensor[_{\sigma}]{A}{}$ is a $(\tensor[_{\sigma}]{K}{},H)$-bicomodule and the convolution invertibility of $\sigma$. We have also dropped the symbol $\Sigma$ for the right and left $H$-coactions in the equations to simplify the expression. Since $\phi_{A,A}$ is an isomorphism, we also have that the opposite implication is true.
	\end{proof}
	Then, for the deformed principal $H$-comodule algebra $\tensor[_{\sigma}]{B}{}\subseteq\tensor[_{\sigma}]{A}{}$ with external symmetry $K_{\sigma}$ the strong connection is given by
	\beq
	\label{eq:strong_conn_left}
	\tensor[_{\sigma}]{\ell}{}(h)=\sigma^{-1}\left(\tensor{h}{^{\1}_{(-1)}},\tensor{h}{^{\2}_{(-1)}}\right)\tensor{h}{^{\1}_{(0)}}\tensor[^{\sigma}]{\otimes}{}\tensor{h}{^{\2}_{(0)}},\quad h\in H.
	\eeq
	
	If we now consider a $H$-cleft extension $B\subseteq A$ with left $B$-linear and right $H$-colinear isomorphism $\theta:B\otimes H\longrightarrow A$, we have that the deformed extension $\tensor[_{\sigma}]{B}{}\subseteq\tensor[_{\sigma}]{A}{}$ is $H$-cleft if the isomorphism $\theta$ is also left $K$-colinear. In general, this might not be true, and in that case, one has to apply to $\theta$ the map $\fS$ of Proposition 3.17 in \cite{Aschieri_2016} and check if $\fS(\theta)\in\tensor[_{\tensor[_{\sigma}]{B}{}}]{\Hom}{^{H}}(\tensor[_{\sigma}]{B}{}\otimes H,\tensor[_{\sigma}]{A}{})$ is invertible. In general, a strong connection in this case is given by the combination of \eqref{eq:strong_conn_cleft} and \eqref{eq:strong_conn_left}
	
	\begin{equation}
		\label{eq:strong_conn_deformed_left}
		\tensor[_{\sigma}]{\ell}{}(h)=\sigma^{-1}\left(\phi^{-1}(h_{(1)})_{(-1)},\phi(h_{(2)})_{(-1)}\right)\phi^{-1}(h_{(1)})_{(0)}\tensor[^{\sigma}]{\otimes}{}\phi(h_{(2)})_{(0)},\quad h\in H.
	\end{equation}

	We close this section by finding  a strong connection when we deform a principal comodule algebra with both a $2$-cocycle of the structure Hopf algebra and a $2$-cocycle coming from an external symmetry.
	For a principal $H$-comodule algebra with external symmetry $K$, it was shown in \cite{Aschieri_2016} that if we first deform using $\gamma:H\otimes H\longrightarrow\bK$ and then $\sigma:K\otimes K\longrightarrow\bK$ we obtain the same result as performing the deformation in the reverse order.
	
	We now show that the deformed strong connection also shares the same feature
	\begin{thm}
		The algebra $B\subseteq A$ is a principal $H$-comodule algebra with external symmetry $K$ if and only if $\tensor[_{\sigma}]{B}{}\subseteq\tensor[_{\sigma}]{A}{_{\gamma}}$ is a principal $H_{\gamma}$-comodule algebra with externa symmetry $\tensor[_{\sigma}]{K}{}$.
	\end{thm}
	\begin{proof}
		We have to write down the same diagram of \ref{thm:deformed_strong_conn_right} and \ref{thm:strong_conn_left}, in this context they are         \begin{equation*}
			\begin{tikzcd}
				\underline{H_{\gamma}} \arrow[rr, "_{\sigma}(\ell_{\gamma})"] \arrow[rrdd, "\Sigma(\ell_{\gamma})"'] & & \tensor[_{\sigma}]{A}{_{\gamma}}\tensor[^{\sigma}]{\otimes}{^{\gamma}}\tensor[_{\sigma}]{A}{_{\gamma}} \arrow[dd, "\phi_{A_{\gamma},A_{\gamma}}"]\\
				& & \\
				& & _{\sigma}(A_{\gamma}\otimes^{\gamma}A_{\gamma})
			\end{tikzcd}, \quad \begin{tikzcd}
				\underline{H}_{\gamma} \arrow[rr, "\Gamma(_{\sigma}\ell)"] \arrow[dd, "\ff^{-1}"'] & & (_{\sigma}A\tensor[^{\sigma}]{\otimes}{} _{\sigma}A)_{\gamma} \arrow[dd, "\alpha^{-1}_{_{\sigma}A,_{\sigma}A}"]\\
				& &\\
				\underline{H_{\gamma}} \arrow[rr, "(_{\sigma}\ell)_{\gamma}"] & & \tensor[_{\sigma}]{A}{_{\gamma}}\tensor[^{\sigma}]{\otimes}{^{\gamma}}\tensor[_{\sigma}]{A}{_{\gamma}}
			\end{tikzcd}
		\end{equation*}
		We respectively obtain
		\begin{equation*}
			_{\sigma}(\ell_{\gamma}):=\phi^{-1}_{A_{\gamma},A_{\gamma}}\circ\Sigma(\ell_{\gamma}),\quad (_{\sigma}\ell)_{\gamma}:=\alpha^{-1}_{_{\sigma}A,_{\sigma}A}\circ\Gamma(_{\sigma}\ell)\circ\ff,
		\end{equation*}
		the explicit form of the first is given by
		\begin{equation*}
			_{\sigma}(\ell_{\gamma})(h)=\sigma^{-1}\left(\tensor{h}{_{(2)}^{\1}_{(-1)}},\tensor{h}{_{(2)}^{\2}_{(-1)}}\right)\tensor{h}{_{(2)}^{\1}_{(0)}}\tensor[^{\sigma}]{\otimes}{^{\gamma}}\tensor{h}{_{(2)}^{\2}_{(0)}}u_{\gamma}(h_{(1)}),
		\end{equation*}
		for any $h\in\underline{H_{\gamma}}$.
		The second one gives us
		\begin{align*}
			(_{\sigma}\ell)_{\gamma}(h)&=\sigma^{-1}\left(\tensor{h}{_{(3)}^{\1}_{(-1)}},\tensor{h}{_{(3)}^{\2}_{(-1)}}\right)\tensor{h}{_{(3)}^{\1}_{(0)(0)}}\tensor[^{\sigma}]{\otimes}{^{\gamma}}\tensor{h}{_{(3)}^{\2}_{(0)(0)}}u_{\gamma}(h_{(1)})\\
			&\gamma\left(\tensor{h}{_{(3)}^{\1}_{(0)(1)}},\tensor{h}{_{(3)}^{\2}_{(0)(1)}}\right)\gamma^{-1}(S(h_{(2)},h_{(4)})\\
			&=\sigma^{-1}\left(\tensor{h}{_{(4)}^{\1}_{(-1)}},\tensor{h}{_{(4)}^{\2}_{(0)(-1)}}\right)\tensor{h}{_{(4)}^{\1}_{(0)}}\tensor[^{\sigma}]{\otimes}{^{\gamma}}\tensor{h}{_{(4)}^{\2}_{(0)(0)}}u_{\gamma}(h_{(1)})\\
			&\gamma\left(S(h_{(3)}),\tensor{h}{_{(4)}^{\2}_{(1)}}\right)\gamma^{-1}(S(h_{(2)},h_{(5)})\\
			&=\sigma^{-1}\left(\tensor{h}{_{(4)}^{\1}_{(-1)}},\tensor{h}{_{(4)}^{\2}_{(-1)}}\right)\tensor{h}{_{(4)}^{\1}_{(0)}}\tensor[^{\sigma}]{\otimes}{^{\gamma}}\tensor{h}{_{(4)}^{\2}_{(0)}}u_{\gamma}(h_{(1)})\\
			&\gamma\left(S(h_{(3)}),h_{(5)}\right)\gamma^{-1}\left(S(h_{(2)},h_{(6)}\right)\\
			&=\sigma^{-1}\left(\tensor{h}{_{(3)}^{\1}_{(-1)}},\tensor{h}{_{(3)}^{\2}_{(-1)}}\right)\tensor{h}{_{(3)}^{\1}_{(0)}}\tensor[^{\sigma}]{\otimes}{^{\gamma}}\tensor{h}{_{(3)}^{\2}_{(0)}}u_{\gamma}(h_{(1)})\\
			&(\gamma*\gamma^{-1})\left(S(h_{(2)}),h_{(4)}\right)\\
			&=\sigma^{-1}\left(\tensor{h}{_{(2)}^{\1}_{(-1)}},\tensor{h}{_{(2)}^{\2}_{(-1)}}\right)\tensor{h}{_{(2)}^{\1}_{(0)}}\tensor[^{\sigma}]{\otimes}{^{\gamma}}\tensor{h}{_{(2)}^{\2}_{(0)}}u_{\gamma}(h_{(1)}),
		\end{align*}
		for any $h\in\underline{H_{\gamma}}$, then we have $(_{\sigma}\ell)_{\gamma}= \tensor[_{\sigma}]{(\ell_{\gamma})}{}$. The same computation in the proofs of \ref{thm:deformed_strong_conn_right} and \ref{thm:strong_conn_left} lead us to conclusion that $(_{\sigma}\ell)_{\gamma}$, and consequently $\tensor[_{\sigma}]{(\ell_{\gamma})}{}$, is a strong connection for $\tensor[_{\sigma}]{B}{}\subseteq\tensor[_{\sigma}]{A}{_{\gamma}}$ if $\ell$ is, and because the maps involved in its definition are all invertible we conclude that the reverse implication holds too.
	\end{proof}
	
	This last result allows us to write $\tensor[_{\sigma}]{\ell}{_{\gamma}}$ for a strong connection of $\tensor[_{\sigma}]{B}{}\subseteq\tensor[_{\sigma}]{A}{_{\gamma}}$ with no risk of confusion
	\begin{equation}
		\label{eq:strong_conn_both_deformed}
		\tensor[_{\sigma}]{\ell}{_{\gamma}}(h)=\sigma^{-1}\left(\tensor{h}{_{(2)}^{\1}_{(-1)}},\tensor{h}{_{(2)}^{\2}_{(-1)}}\right)\tensor{h}{_{(2)}^{\1}_{(0)}}\tensor[^{\sigma}]{\otimes}{^{\gamma}}\tensor{h}{_{(2)}^{\2}_{(0)}}u_{\gamma}(h_{(1)}),\quad h\in H_{\gamma}.
	\end{equation}

	\section{Characteristic classes and naturality}
	In this section, we recall the basic definitions of cyclic homology for algebras and review the theory of the noncommutative Chern-Weil homomorphism for Hopf-Galois extensions, as introduced in \cite{hajac2021cyclic}. We then introduce characteristic classes for principal comodule algebras and prove that they are natural transformations generalizing the classic case.    
	
	\subsection{Hochschild and cyclic homology}
	In its full generality \cite{loday2013cyclic}, Hochschild homology is defined for any bimodule over any algebra $B$. For our purpose, the case where the bimodule is the algebra $B$ itself is enough. The motivation for this homology theory stems from the Hochschild-Kostant-Rosenberg theorem, which characterizes Hochschild homology for commutative algebras.
	
	Given an algebra $B$, define the complex $C_n(B):=B^{\otimes(n+1)}$. Moreover, introduce the \textit{face operators} $d_i:C_n(B)\longrightarrow C_{n-1}(B)$
	\begin{align*}
		d_0(b_0\otimes\dots\otimes b_n)&:=b_0b_1\otimes\dots\otimes b_n,\\
		d_i(b_0\otimes\dots\otimes b_n)&:=b_0\otimes\dots\otimes b_ib_{i+1}\otimes\dots\otimes b_n,\quad 1\leq i\leq n\\
		d_n(b_0\otimes\dots\otimes b_n)&:=b_nb_0\otimes\dots\otimes b_{n-1}.
	\end{align*}
	One has that $d:=\sum_{i=0}^{n}(-1)^id_i$ satisfies $d\circ d=0$. Thus, we have the sequence
	\begin{equation*}
		\dots\xrightarrow[]{d} C_n(B)\xrightarrow[]{d} C_{n-1}(B)\xrightarrow[]{d} C_{n-2}(B)\xrightarrow[]{d}\dots\xrightarrow{d}C_0(B), 
	\end{equation*}
	with the corresponding homology theory defined by the face operator $d$, which is called the \textit{Hochschild homology} of $B$ and denoted by $HH_*(B)$.
	
	There is a natural action of the cyclic group $\bZ_{n+1}$ on $C_n(B)$ given by
	\begin{equation}
		\label{eq:cyclic_action}
		t_n(b_0\otimes\dots\otimes b_n)=(-1)^n(b_n\otimes b_0\dots\otimes b_{n-1}).
	\end{equation}
	the generator $t_n\in\bZ_{n+1}$. It is straightforward to check that $t^{n+1}=\id$. The $\bZ_{n+1}$-invariant elements in $C_n(B)$ are called \textit{cyclic tensor}, and we denote them by $C^{\delta}_n(B):=C_n(B)/\mathrm{Ker}(\delta)$ where $\delta:=\id-t_n$. There are the \textit{norm operator} $N:=1+t+\dots+t^n:C_n(B)\longrightarrow C_n(B)$ and the \textit{truncated operator} $d':=\sum_{i=0}^{n-1}(-1)^id_i:C_n(B)\longrightarrow C_{n-1}(B)$ that satisfy
	\begin{equation*}
		(\id-t)d'=d(\id-t),\quad d'N=Nd,
	\end{equation*}
	With this, the \textit{Connes complex} $(C^{\delta}_*(B),d)$ has sequence
	\begin{equation*}
		\dots\xrightarrow{d}C^{\delta}_n(B)\xrightarrow{d}C^{\delta}_{n-1}(B)\xrightarrow{d}C^{\delta}_*(B)\xrightarrow{d}\dots\xrightarrow{d}C^{\delta}_0(B),
	\end{equation*}
	and homology theory
	\begin{equation*}
		HC_n(B):=\frac{\mathrm{Ker}(d:C^{\delta}_n(B)\longrightarrow C^{\delta}_{n-1}(B))}{\mathrm{Im}(d:C^{\delta}_{n+1}(B)\longrightarrow C^{\delta}_n(B))},\quad HC(B):=\oplus_{n=0}^{\infty}HC_n(B)
	\end{equation*}
	called the \textit{cyclic homology} of $B$.
	
	%	Hochschild and cyclic homology are related via a long exact sequence discovered by Connes who introduced the so-called \textit{periodicity operator} $P:HC_n(B)\longrightarrow HC_{n-2}(B)$, To write down explicitly $P$, we need to introduce some tools. First of all, if $x\in C_n(B)$ we denote by $\bar{x}$ the corresponding element in $C_n^{\delta}(B)$. Moreover, we define the map
	%	\begin{equation*}
	%		d^{[2]}:=\sum_{i,j=0}^n(-1)^{i+j}d_id_j:C_n(B)\longrightarrow C_{n-2}(B).
	%	\end{equation*}
	%	At this point, the periodicity operator has the following expression
	%%		\label{eq:periodicity_op}
	%		P([\bar{x}]):=-\frac{1}{n(n-1)}[\overline{d^{[2]}(x)}],
	%	\end{equation}
	%	where $[\bar{x}]\in HC_n(B)$ denotes the homology class of $\bar{x}$.

	\subsection{The Chern-Weil homomorphism for principal comodule algebras}
	Given a Hopf algebra $H$ we have the following vector space
	\begin{equation}
		\label{eq:cotarces}
		H^{tr}:=\{h\in H|h_{(1)}\otimes h_{(2)}=h_{(2)}\otimes h_{(1)}\},
	\end{equation}
	that we call the space of \textit{cotraces}. The name is justified by the fact that any element of $H^{tr}$ defines a trace in the space $\mathrm{Hom}(H,\bK)$, in fact for $h\in H^{tr}$ one has
	\begin{equation}
		\label{eq:trace_of_cotrace}
		\tau_h:\mathrm{Hom}(H,\bK)\longrightarrow \bK,\quad f\longmapsto f(h),
	\end{equation}
	from which one finds that $\tau_h(f*g)=\tau_h(g*f)$ for all $f,g\in\mathrm{Hom}(H,\bK)$. For the Hopf algebra $O(G)$ of the regular functions of an linear algebraic group $G$ we have the identification
	\begin{equation*}
		\begin{split}
			O(G)^{tr}&=\{f\in O(G)|f(gh)=f(hg),\quad \forall g,h\in G\}\\
			&=\{f\in O(G)|f(hgh^{-1})=f(g),\forall g,h\in G\}\\
			&=O(\mathrm{Ad}(G))^G,
		\end{split}   
	\end{equation*}
	with the algebra of adjoint invariant elements $O(\mathrm{Ad}(G))^G$. Via a filtration using the ideal $\mathrm{ker}(\epsilon)$, this algebra gives the algebra $\bK[\fg]^G$ of adjoint-invariant polynomials on the Lie algebra $\fg$.
	
	Cotraces can be identified with the space $H^{tr}\simeq H\square^{H\otimes H^{op}}(\overbrace{H\square^H\dots\square^HH}^{n+1})$ for any $n\in\mathrm{N}$ as shown in \cite[Lemma 4.2]{hajac2021cyclic}.   
	
	Consider now a principal $H$-comodule algebra $B\subseteq A$. The cotensor product $A\square^HA$ and $(A\otimes A)^{coH}$ are isomorphic as vector spaces  \cite[Lemma 3.1]{schneider1990principal}. For the rest of the section, we use the short notation $M:=A\square^HA$. Recall that the latter has a $B$-coring structure \cite{brz-wis} with counit given by
	\beq
	\label{eq:counit_coring}
	\underline{\epsilon}(a\otimes\tilde{a}):=a\tilde{a},
	\eeq
	for all $a\otimes\tilde{a}\in M$. Because this map is left $B$-linear we can define multiplication in the following way 
	\begin{equation}
		\label{eq:multiplication_row_ext}
		m\cdot m':=\underline{\epsilon}(m)m'.
	\end{equation}
	For the Chern-Weil theory, one only needs that $(M,\underline{\epsilon})$ is a left $B$-module.
	
	%\begin{remark}
	%For a principal comodule algebra $B\subseteq A$, the space $M$ can be also endowed with a $B$-bialgebroid structure \cite{brz-wis}. The algebra structure in this case is inherited by $A\otimes A^{op}$, i.e. $(a\otimes\tilde{a})(a'\otimes\tilde{a}'):=aa'\otimes \tilde{a}'\tilde{a}$. The product \eqref{eq:multiplication_row_ext} that we consider here is in general different.	
	%\end{remark}
	Collecting the results from Lemma 4.3 to 4.6 in \cite{hajac2021cyclic}, we have that, given a strong connection $\ell:H\longrightarrow A\otimes A$ a map
	\begin{equation}
		\label{eq:c_n_1}
		c_n(\ell):H^{tr}\longrightarrow M^{\otimes (n+1)},\quad h\longmapsto l(h_{(1)})\otimes\dots\otimes l(h_{(n+1)}), 
	\end{equation}
	is induced for any $n\in\mathbb{N}$. The element
	\begin{equation}
		\label{eq:c_n_2}
		c_n(\ell)(h):=\left(\tensor{h}{_{(n+1)}^{\2}}\otimes\tensor{h}{_{(1)}^{\1}}\right)\otimes\dots\otimes\left(\tensor{h}{_{(n)}^{\2}}\otimes\tensor{h}{_{(n+1)}^{\1}}\right),
	\end{equation}
	is cyclic-symmetric in $M^{\otimes(n+1)}$ and for any face operator $d_i$ with $i=0,\dots,n$ one has
	\begin{equation}
		\label{eq:cyclic_op_on_c_n}
		d_ic_n(\ell)(h)=c_{n-1}(\ell)(h),\quad \forall h\in H^{tr}
	\end{equation}  
	
	Now let $HC_*(M)$ be the cyclic homology of the algebra $M$ with the product \eqref{eq:multiplication_row_ext}. The identities \eqref{eq:c_n_2} and \eqref{eq:cyclic_op_on_c_n} allow us to define a $2n$-cycle in $HC_{2n}(M)$ explicitly given by
	\beq
	\label{eq:Chern-Weil_map}
	\widetilde{\mathrm{chw}}_n(\ell)(h):=\sum_{i=0}^{2n}(-1)^{\lfloor\frac{i}{2}\rfloor}\frac{i!}{\lfloor\frac{i}{2}\rfloor!}c_i(\ell)(h),\quad h\in H^{tr}.
	\eeq
	The fact that $\widetilde{\mathrm{chw}}_n(l)(h)$ defines a $2n$-cycle follows from the identities
	\begin{align}
		\label{eq:cyclic_Chern1}
		d(2c_{2n}(\ell)(h))&=(1-t)c_{2n-1}(\ell)(h),\\
		\label{eq:cyclic_Chern2}
		d'(nc_{2n-1}(\ell)(h))&=Nc_{2n-2}(\ell)(h),
	\end{align}
	for all $h\in H^{tr}$. 
	
	%    Moreover, it is stable under the periodicity operator \eqref{eq:periodicity_op}.
	
	We can now use the counit of $M$ to induce a map from $M^{\otimes(n+1)} $ into $B^{\otimes(n+1)}$, this is done by applying it to every factor of the tensor product. For any $n\in\mathbb{N}$ one has the formula for any $h\in H^{tr}$
	\beq
	\label{eq:c_n_base_algebra}
	x_n(\ell,h):=\left(\underline{\epsilon}^{\otimes(n+1)}\circ c_n(\ell)\right)(h)=\tensor{h}{_{(n+1)}^{\2}}\tensor{h}{_{(1)}^{\1}}\otimes\dots\otimes\tensor{h}{_{(n)}^{\2}}\tensor{h}{_{(n+1)}^{\1}}\in B^{\otimes(n+1)}.
	\eeq
	Thus, we have the following
	\begin{defn}
		\label{defn:noncommutative_Chern-Weil}
		For any principal $H$-comodule algebra $B\subseteq A$ with strong connection $\ell$, the \textbf{Chern-Weil homomorphism} is given by
		\begin{equation*}
			\mathrm{chw}_n(\ell):H^{tr}\longrightarrow\mathrm{HC}_{2n}(B),\quad h\longmapsto \sum_{i=0}^{2n}(-1)^{\lfloor\frac{i}{2}\rfloor}\frac{i!}{\lfloor\frac{i}{2}\rfloor!}x_i(\ell,h),
		\end{equation*}    
	\end{defn}

	\subsection{Pullbacks and naturality}
	In this subsection, which closes the chapter, we study how the Chern-Weil homomorphism behaves with respect to a noncommutative pullback of principal comodule algebras. We start by stating and proving the following
	\begin{prop}
		\label{prop:pushforward}
		Let $B\subseteq A$ be a principal $H$-comodule algebra, $\bar{A}$ a right $H$-comodule algebra with coaction invariant elements $\bar{B}$, and $f:A\longrightarrow \bar{A}$ a unital right $H$-comodule algebra morphism. Then the have
		\begin{enumerate}
			\item $\bar{B}\subseteq\bar{A}$ is a principal $H$-comodule algebra,
			\item There is a left $\bar{B}$-linear and right $H$-colinear isomorphism $\bar{A}\simeq \bar{B}\otimes_BA$,
			\item If $B\subseteq A$ is a cleft extension so is $\bar{B}\subseteq\bar{A}$ is.
		\end{enumerate}
	\end{prop}
	\begin{proof}
		\begin{enumerate}
			\item One can easily check that for any $H$-colinear morphism $f(B)\subseteq\bar{B}$. We prove that the map
			\begin{equation*}
				\bar{\ell}:=(f\otimes f)\circ \ell:H\longrightarrow \bar{A}\otimes \bar{A}
			\end{equation*}
			is a strong connection for $\bar{B}\subseteq \bar{A}$ by checking the identities \eqref{eq:strong_conn1}-\eqref{eq:strong_conn3}. Let $h$ be any element of $H$:
			\begin{equation*}
				\begin{split}
					[(\mathrm{id}_{\bar{A}}\otimes\bar{\rho})\circ \bar{\ell}](h)&=f(h^{\1})\otimes f(h^{\2})_{(0)}\otimes f(h^{\2})_{(1)}\\
					&=f(h^{\1})\otimes f(\tensor{h}{^{\2}_{(0)}})\otimes \tensor{h}{^{\2}_{(1)}}\\
					&=f(\tensor{h}{_{(1)}^{\1}})\otimes f(\tensor{h}{_{(1)}^{\2}})\otimes h_{(2)}=[(\bar{\ell}\otimes\mathrm{id}_H)\circ\Delta](h),
				\end{split}
			\end{equation*}
			where in the second line we used that $f$ is a right $H$-comodule morphism and in the third one equation \eqref{eq:strong_conn_Sw1}. Moving on, we find
			\begin{equation*}
				\begin{split}
					[(\lambda_{\bar{A}}\otimes\mathrm{id}_{\bar{A}})\circ \bar{\ell}](h)&=S^{-1}(f(h^{\1})_{(1)})\otimes f(h^{\1})_{(0)}\otimes f(h^{\2})\\
					&=S^{-1}(S(h_{(1)}))\otimes f(\tensor{h}{_{(2)}}^{\1})\otimes f(\tensor{h}{_{(2)}}^{\2})\\
					&=h_{(1)}\otimes f(\tensor{h}{_{(2)}}^{\1})\otimes f(\tensor{h}{_{(2)}}^{\2})=[(\mathrm{id}_H\otimes \bar{\ell})\circ\Delta](h), 	
				\end{split}
			\end{equation*}
			we used again that $f$ is a right $H$-comodule morphism and in the third line equation \eqref{eq:strong_conn_Sw2}. Finally, we have
			\begin{equation*}
				\begin{split}
					(\can_{\bar{A}}\circ\pi_B\circ\bar{\ell})(h)&=f(h^{\1})f(h^{\2})_{(0)}\otimes f(h^{\2})_{(1)}\\
					&=f(h^{\1})f(\tensor{h}{^{\2}_{(0)}})\otimes\tensor{h}{^{\2}_{(1)}}\\
					&=f(h^{\1}\tensor{h}{^{\2}_{(0)}})\otimes \tensor{h}{^{\2}_{(1)}}=1_{\bar{A}}\otimes h.
				\end{split}
			\end{equation*}
			We used that $f$ is a $H$-colinear algebra morphism and the definition of translation map.
			\item $\bar{B}$ is a $B$-bimodule via the map $f|_B$, hence $\bar{B}\otimes_B A$ is well-defined and inside of it we have the identification
			\begin{equation*}
				\bar{b}f(b)\otimes_B a=\bar{b}\otimes_Bba,
			\end{equation*}
			for any $\bar{b}\in \bar{B}$, $b\in B$ and $a\in A$. Define the linear maps
			\begin{align*}
				\alpha:\bar{A}\longrightarrow \bar{B}\otimes_B A&,\quad \bar{a}\longmapsto \bar{a}_{(0)}f(\tensor{\bar{a}}{_{(1)}^{\1}})\otimes_B\tensor{\bar{a}}{_{(1)}^{\2}}\\
				\beta:\bar{B}\otimes_B A\longrightarrow\bar{A}&,\quad \bar{b}\otimes_B a\longmapsto \bar{b}f(a)
			\end{align*}
			it is easy to check that they are left $\bar{B}$-module and right $H$-comodule morphism. $\bar{A}$ lands into $\bar{B}\otimes_B A$ through $\alpha$
			\begin{equation*}
				\begin{split}
					(\bar{\rho}\otimes_B\mathrm{id}_A)(\alpha(\bar{a}))&=\bar{a}_{(0)}f(\tensor{\bar{a}}{_{(2)}^{\1}}_{(0)})\otimes\bar{a}_{(1)}\tensor{\bar{a}}{_{(2)}^{\2}_{(1)}}\otimes_B\tensor{\bar{a}}{_{(2)}^{\2}}\\
					&=\bar{a}_{(0)}f(\tensor{\bar{a}}{_{(2)(2)}^{\1}})\otimes\bar{a}_{(1)}S(\tensor{\bar{a}}{_{(2)(1)}})\otimes_B\tensor{\bar{a}}{_{(2)(2)}^{\2}}\\
					&=\bar{a}_{(0)}f(\tensor{\bar{a}}{_{(3)}^{\1}})\otimes\bar{a}_{(1)}S(\tensor{\bar{a}}{_{(2)}})\otimes_B\tensor{\bar{a}}{_{(3)}^{\2}}\\
					&=\bar{a}_{(0)}f(\tensor{\bar{a}}{_{(1)}^{\1}})\otimes1_H\otimes_B\tensor{\bar{a}}{_{(1)}^{\2}},
				\end{split}
			\end{equation*}
			we used the fact that $f$ is a right $H$-comodule morphism, the coassociativity of the coproduct, equation \eqref{eq:strong_conn_Sw2} and the antipode equation $h_{(1)}S(h_{(2)})=\epsilon(h)$ for all $h\in H$.

			We are now ready to prove that $\alpha$ and $\beta$ are the inverse of each other:
			\begin{equation*}
				\begin{split}
					(\alpha\circ\beta)(\bar{b}\otimes_B a)&=\bar{b}f(a_{(0)})f(\tensor{a}{_{(1)}^{\1}})\otimes_B\tensor{a}{_{(1)}^{\2}}=\bar{b}f(a_{(0)}\tensor{a}{_{(1)}^{\1}})\otimes_B\tensor{a}{_{(1)}^{\2}}=\bar{b}\otimes_B a,\\
					(\beta\circ\alpha)(\bar{a})&=\bar{a}_{(0)}f(\tensor{\bar{a}}{_{(1)}^{\1}})f(\tensor{\bar{a}}{_{(1)}^{\2}})=\bar{a}_{(0)}f(\tensor{\bar{a}}{_{(1)}^{\1}}\tensor{\bar{a}}{_{(1)}^{\2}})=\bar{a}_{(0)}\varepsilon(\bar{a}_{(1)})=\bar{a}.
				\end{split}
			\end{equation*}
			\item If $\phi:H\longrightarrow A$ is a cleaving map of $B\subseteq A$ then $\bar{\phi}:=f\circ\phi:H\longrightarrow\bar{A}$ is a cleaving map for $\bar{B}\subseteq\bar{A}$. In fact, $\bar{\phi}$ is right $H$-colinear because it is a composition of right $H$-colinear morphisms, it is convolution invertible with inverse $\bar{\phi}^{-1}=f\circ\phi^{-1}$, and unital because both $f$ and $\phi$ are.
		\end{enumerate}		
	\end{proof}
	
	\begin{remark}
		This result is the noncommutative analogue of the pullback principal bundle via a bundle map. If $P\longrightarrow P/G$ is a principal $G$-bundle and $F:P'\longrightarrow P$ is a $G$-equivariant smooth map, then $P'$ is a $G$-space such that the action is free and transitive on the fibers of $P'\longrightarrow P'/G$. Moreover $F/G:P'/G\longrightarrow P/G$ is a smooth map on the base spaces such that $(F,F/G)$ is a principal $G$-bundle morphism. One can show that there is a diffeomorphism $P'\simeq P'/G\times_{P/G}P$ \cite{hajac2018pullbacks}. 
	\end{remark}
	
	Let now $\cQ b_H(B)$ denote the category of isomorphism classes of principal $H$-comodule algebras over $B$. Its morphisms are left $B$-linear and right $H$-colinear algebra morphisms. We can rephrase Proposition \ref{prop:pushforward} by saying that any right $H$-comodule algebra morphism $f:A\longrightarrow\bar{A}$ restricts and corestics to an algebra morphism $f:B\longrightarrow\bar{B}$ and $\cQ b_H(f):\cQ b_H(B)\longrightarrow\cQ b_H(\bar{B})$ maps the isomorphism class of $A$ to that of $\bar{B}\otimes_B A$. In other words, we have the following
	
	\begin{cor}
		\label{cor:quantum_princ_functor}
		$\cQ b_H(-)$ is a covariant functor from the category of algebras to the category of isomorphism classes of principal $H$-comodule algebras.
	\end{cor}
	\begin{proof}
		The first and second statements, along with their proofs in \ref{prop:pushforward}, prove the claim.
	\end{proof}
	
	We are now ready to prove the functoriality of the Chern-Weil homomorphism for principal comodule algebras. First of all, recall that if $A$ and $A'$ are principal $H$-comodules algebras over $B$ and $g:A\longrightarrow A'$ is a left $B$-linear right $H$-colinear algebra morphism, then $g$ is invertible and restricts to the identity on $B$\cite{schneider1990principal}. Moreover, from the strong connection formula in the proof \ref{prop:pushforward} and \eqref{eq:c_n_base_algebra} we find
	\begin{align*}
		x_n(\ell',h)&=g(\tensor{h}{_{(n+1)}^{\2}}\tensor{h}{_{(1)}^{\1}})\otimes\dots\otimes g(\tensor{h}{_{(n)}^{\2}}\tensor{h}{_{(n+1)}^{\1}})\\
		&=\tensor{h}{_{(n+1)}^{\2}}\tensor{h}{_{(1)}^{\1}}\otimes\dots\otimes\tensor{h}{_{(n)}^{\2}}\tensor{h}{_{(n+1)}^{\1}}\\
		&=x_n(\ell,h)
	\end{align*}
	for any $n\in\bN$ and $h\in H^{tr}$. Thus, isomorphic principal comodule algebras have the same Chern-Weil homomorphisms. For any $h\in H^{tr}$ and any fixed $B$, we can define a map 
	\begin{equation}
		\label{eq:characteristic_class}
		c_B:\cQ b_H(B)\longrightarrow HC_{2n}(B),\quad \left[A\right]\longmapsto\chw_n(\ell_A)(h)
	\end{equation}
	that assigns to the class of $A$ the homology class $\chw_n(\ell_A)(h)$ of \eqref{eq:Chern-Weil_map}. We refer to these maps as \textit{characteristic classes} of principal comodule algebras. For instance if $h=c_{\varphi}\in H^{tr}$ is the cotrace coming form a finite-dimensional corepresentation of $H$ as in \cite{hajac2021cyclic} then $c_B$ is the Chern-Galois character \cite{brzezinski2004chern}. It is interesting to notice that cotraces of Hopf algebras are rare to find in the literature, so it is our goal for future work to find other elements in $H^{tr}$ that allow to define noncommutative characteristic classes.
	
	For the map we just introduced, we have the following
	\begin{prop}
		\label{prop:pushforward_Chern-Weil}
		The map \eqref{eq:characteristic_class} induced by the Chern-Weil homomorphism of principal comodule algebras is a natural transformation from the functor $\cQ b_H(-)$ to the cyclic homology functor $HC(-)$. 
	\end{prop}
	\begin{proof}
		We have to prove that the following diagram 
		\begin{equation*}
			\begin{tikzcd}
				\cQ b_H(B) \arrow[rr, "\cQ b_{H}(f)"] \arrow[dd, "c_{B}"'] & & \cQ b_H(\bar{B}) \arrow[dd, "c_{\bar{B}}"]\\
				& & \\
				HC_{2n}(B) \arrow[rr, "HC_{2n}(f)"] & & HC_{2n}(\bar{B}) 
			\end{tikzcd}
		\end{equation*}
		is commutative for any $B$, $\bar{B}$ and $f:A\longrightarrow \bar{A}$ as in  \ref{prop:pushforward}.
		We know from \ref{prop:pushforward} the form of $\bar{\ell}$ and that $c_{\bar{B}}$ is determined by the element  $x_n(\bar{\ell},h)$ for $h\in H^{tr}$ in \eqref{eq:c_n_base_algebra}, so 
		\begin{equation*}
			\begin{split}
				x_n(\bar{\ell},h)&=f(\tensor{h}{_{(n+1)}^{\2}})f(\tensor{h}{_{(1)}^{\1}})\otimes\dots\otimes f(\tensor{h}{_{(n)}^{\2}})f(\tensor{h}{_{(n+1)}^{\1}})\\
				&=f(\tensor{h}{_{(n+1)}^{\2}}\tensor{h}{_{(1)}^{\1}})\otimes\dots\otimes f(\tensor{h}{_{(n)}^{\2}}\tensor{h}{_{(n+1)}^{\1}})=HC(f)(x_n(\ell,h)),
			\end{split}
		\end{equation*}
		for any $n\in\bN$. We can then conclude that $HC(f)\circ c_B=c_{\bar{B}}\circ\cQ b_H(f)$.
	\end{proof}
	
	For a $H$-cleft extension $B\subseteq A$ with cleaving map $\phi$, the map $c_B:\cQ b_H(B)\longrightarrow HC_{2n}(B)$ is determined by $x_n(\ell,h)$ where $\ell$ is the strong connection in \eqref{eq:strong_conn_cleft}. It is easy to check that
	\begin{equation}
		\label{eq:Chern-Weil_cleft}
		x_n(\ell,h)=\epsilon(h) \underbrace{1\otimes\dots\otimes 1}_{n+1},\quad \forall h\in H^{tr}.
	\end{equation}
	Thus, the $n^{th}$-Chern-Weil map of \eqref{eq:Chern-Weil_map} is the expression above multiplied by the constants $C_n=\sum_{i=0}^{2n}(-1)^{\lfloor\frac{i}{2}\rfloor}\frac{i!}{\lfloor\frac{i}{2}\rfloor!}$. Thus, putting together the third statement of \ref{prop:pushforward} and Proposition \ref{prop:pushforward_Chern-Weil}, for any $H$-comodule algebra morphism $f:A\longrightarrow\bar{A}$, with $B\subseteq A$ $H$-cleft as above, we have that
	\begin{equation*}
		\chw_n(\bar{\ell},h)=\epsilon(h)C_n\underbrace{1\otimes\dots\otimes1}_{n+1},\quad \forall h\in H^{tr}.
	\end{equation*}
	
	\section{Chern-Weil homomorphism for deformed extensions}
	In this section, we study the behavior of the Chern-Weil homomorphism under $2$-cocycle deformations of principal comodule algebras. We show that the homomorphism is not affected by a $2$-cocycle deformation coming from the structure Hopf algebra of the extension, so in this case it is a deformation invariant; meanwhile, if the cocycle belongs to an external symmetry, it chances accordingly to the formula \eqref{eq:Chern-Weil_map_deformed_ext}. Finally we analyze the combined case which easily follows from the previous ones.
	
	\subsection{Deformation of the structure Hopf algebra} 
	Now let $B\subseteq A$ be a principal $H$-comodule algebra and $\gamma:H\otimes H\longrightarrow\mathbb{K}$ an invertible $2$-cocycle of $H$. Since the deformation $H_{\gamma}$ involves only the product and not the coproduct, we have
	\begin{equation}
		\label{eq:sim_cotraces}
		H_{\gamma}^{tr}\simeq H^{tr},
	\end{equation}
	Moreover, The spaces $M=A\square^H A$ and $M_{\gamma}=A_{\gamma}\square^{H_{\gamma}}A_{\gamma}$ are isomorphic as vector space.
	However, $M$ and $M_{\gamma}$ have different algebra structures. The multiplication on $M_{\gamma}$ is given by the formula
	\begin{equation}
		\label{eq:prod_M_2-co_right}
		m\cdot_{\gamma}m':=\underline{\epsilon}_{\gamma}(m)m',
	\end{equation}
	Explicitly, if we take $a\otimes^{\gamma}\widetilde{a},a'\otimes^{\gamma}\widetilde{a}'\in M_{\gamma}$ we have
	\begin{equation*}
		(a\otimes^{\gamma}\widetilde{a})\cdot_{\gamma}(a'\otimes^{\gamma}\widetilde{a}')=(a\cdot_{\gamma}\widetilde{a})\cdot_{\gamma}a'\otimes^{\gamma}\widetilde{a}'
	\end{equation*}
	where now $\cdot_{\gamma}$ is the multiplication \eqref{eq:deformed_prod_right_algebra}.
	
	Because the subalgebra $B$ does not change under the deformation via a $2$-cocycle $\gamma$ of the structure Hopf algebra ($B_{\gamma}=B$), we have that the cyclic homologies are the same $HC_*(B_\gamma)=HC_*(B)$. Thus, the following diagram is commutative
	\begin{equation*}
		\begin{tikzcd}
			H^{tr}_{\gamma} \arrow[rr, "\chw_n(\ell_{\gamma})"] \arrow[dd, "\simeq"'] & & HC_{2n}(B_{\gamma}) \arrow[dd, "\simeq"]\\
			& & \\
			H^{tr} \arrow[rr, "\chw_n(l)"] & & HC_{2n}(B).
		\end{tikzcd}
	\end{equation*}
	From this we get $\chw_n(\ell_{\gamma})=\chw_n(l)$, where $\ell_{\gamma}$ is the strong connection of \eqref{eq:strong_conn_def_right}.
	
	We now prove this equivalence explicitly by showing how there are cancellations of the $2$-cocycle in the formulas.
	\begin{prop}
		\label{prop:Chern-Weil_deform_right}
		For any $n\in\mathbb{N}$ the map 
		\begin{equation*}
			\begin{split}
				c_n(\ell_{\gamma}):H^{tr}_{\gamma}\longrightarrow M_{\gamma}^{\otimes(n+1)},\\
				h\longmapsto \tensor{h}{_{(2n+2)}^{\2}}\otimes^{\gamma}\tensor{h}{_{(2)}^{\1}}\otimes^{\gamma}\dots\otimes^{\gamma}\tensor{h}{_{(2n)}^{\2}}\otimes^{\gamma}&\tensor{h}{_{(2n+2)}^{\1}}u_{\gamma}(h_{(1)})u_{\gamma}(h_{(3)})\dots u_{\gamma}(h_{(2n+1)}),
			\end{split} 
		\end{equation*}
		is well-defined and its image lies in the cyclic-symmetric part of $M^{\otimes(n+1)}_{\gamma}$. Moreover for any face operator $d_i$ with $i=0,\dots,n$ one has
		\begin{equation*}
			d_ic_n(\ell_{\gamma})=c_{n-1}(\ell_{\gamma}).
		\end{equation*}
	\end{prop}
	\begin{proof}
		We are looking at $M^{\otimes(n+1)}_{\gamma}$ as a circular tensor product, so that for $n=1$ we have that
		\begin{equation*}
			c_1(\ell_{\gamma})(h)=\tensor{h}{_{(4)}^{\2}}\otimes^{\gamma}\tensor{h}{_{(2)}^{\1}}\otimes^{\gamma}\tensor{h}{_{(2)}^{\2}}\otimes^{\gamma}\tensor{h}{_{(4)}^{\1}}u_{\gamma}(h_{(1)})u_{\gamma}(h_{(3)}).
		\end{equation*}
		To check that this element lies in $M^{\otimes 2}_{\gamma}$ we apply first $\rho^{\otimes^{\gamma}}\otimes\mathrm{id}$ and $\mathrm{id}\otimes\rho^{\otimes^{\gamma}}$
		\begin{equation*}
			\begin{split}
				&(\rho^{\otimes^{\gamma}}\otimes^{\gamma}\mathrm{id})(c_1(\ell_{\gamma})(h))=\\
				&=\tensor{h}{_{(4)}^{\2}_{(0)}}\otimes^{\gamma}\tensor{h}{_{(2)}^{\1}_{(0)}}\otimes\tensor{h}{_{(4)}^{\2}_{(1)}}\cdot_{\gamma}\tensor{h}{_{(2)}^{\1}_{(1)}}\otimes\tensor{h}{_{(2)}^{\2}}\otimes^{\gamma}\tensor{h}{_{(4)}^{\1}}u_{\gamma}(h_{(1)})u_{\gamma}(h_{(3)})\\
				&=\tensor{h}{_{(5)}^{\2}}\otimes^{\gamma}\tensor{h}{_{(3)}^{\1}}\otimes\tensor{h}{_{(6)}}\cdot_{\gamma}
				S(h_{(2)})\otimes\tensor{h}{_{(3)}^{\2}}\otimes^{\gamma}\tensor{h}{_{(5)}^{\1}}u_{\gamma}(h_{(1)})u_{\gamma}(h_{(4)})\\
				&=\tensor{h}{_{(5)}^{\2}}\otimes^{\gamma}\tensor{h}{_{(3)}^{\1}}\otimes\tensor{h}{_{(6)}}\cdot_{\gamma}
				S_{\gamma}(h_{(3)})\otimes\tensor{h}{_{(3)}^{\2}}\otimes^{\gamma}\tensor{h}{_{(5)}^{\1}}u_{\gamma}(h_{(1)})u^{-1}_{\gamma}(h_{(2)})u_{\gamma}(h_{(4)})u_{\gamma}(u_{(5)})\\
				&=\tensor{h}{_{(5)}^{\2}}\otimes^{\gamma}\tensor{h}{_{(3)}^{\1}}\otimes\tensor{h}{_{(6)}}\cdot_{\gamma}
				S_{\gamma}(h_{(1)})\otimes\tensor{h}{_{(3)}^{\2}}\otimes^{\gamma}\tensor{h}{_{(5)}^{\1}}u_{\gamma}(h_{(2)})u_{\gamma}(h_{(4)})\\
				&=\tensor{h}{_{(4)}^{\2}}\otimes^{\gamma}\tensor{h}{_{(2)}^{\1}}\otimes1_{H_{\gamma}}\otimes\tensor{h}{_{(2)}^{\2}}\otimes^{\gamma}\tensor{h}{_{(4)}^{\1}}u_{\gamma}(h_{(1)})u_{\gamma}(h_{(3)})
			\end{split}
		\end{equation*}
		where in the second equality we used the properties of the translation map, in the third one the definition of $S_{\gamma}$, then the fact that $u_{\gamma}$ is convolution invertible, in the second to last equality we used the identification $H^{tr}_{\gamma}\simeq H^{tr}$ and Lemma something for $H^{tr}$ and finally the definition of the antipode.
		
		\begin{equation*}
			\begin{split}
				&(\mathrm{id}\otimes^{\gamma}\rho^{\otimes^{\gamma}})(c_1(\ell_{\gamma})(h))=\\
				&=\tensor{h}{_{(4)}^{\2}}\otimes^{\gamma}\tensor{h}{_{(2)}^{\1}}\otimes^{\gamma}\tensor{h}{_{(2)}^{\2}_{(0)}}\otimes^{\gamma}\tensor{h}{_{(4)}^{\1}_{(0)}}\otimes\tensor{h}{_{(2)}^{\2}_{(1)}}\cdot_{\gamma}\tensor{h}{_{(4)}^{\1}_{(1)}}u_{\gamma}(h_{(1)})u_{\gamma}(h_{(3)})\\
				&=\tensor{h}{_{(6)}^{\2}}\otimes^{\gamma}\tensor{h}{_{(2)}^{\1}}\otimes^{\gamma}\tensor{h}{_{(2)}^{\2}}\otimes^{\gamma}\tensor{h}{_{(6)}^{\1}}\otimes\tensor{h}{_{(3)}}\cdot_{\gamma}S(h_{(5)})u_{\gamma}(h_{(1)})u_{\gamma}(h_{(4)})\\
				&=\tensor{h}{_{(8)}^{\2}}\otimes^{\gamma}\tensor{h}{_{(2)}^{\1}}\otimes^{\gamma}\tensor{h}{_{(2)}^{\2}}\otimes^{\gamma}\tensor{h}{_{(8)}^{\1}}\otimes\tensor{h}{_{(3)}}\cdot_{\gamma}S_{\gamma}(h_{(6)})u_{\gamma}(h_{(1)})u_{\gamma}(h_{(4)})u^{-1}_{\gamma}(h_{(5)})u_{\gamma}(h_{(7)})\\
				&=\tensor{h}{_{(6)}^{\2}}\otimes^{\gamma}\tensor{h}{_{(2)}^{\1}}\otimes^{\gamma}\tensor{h}{_{(2)}^{\2}}\otimes^{\gamma}\tensor{h}{_{(6)}^{\1}}\otimes\tensor{h}{_{(3)}}\cdot_{\gamma}S_{\gamma}(h_{(4)})u_{\gamma}(h_{(1)})u_{\gamma}(h_{(5)})\\
				&=\tensor{h}{_{(4)}^{\2}}\otimes^{\gamma}\tensor{h}{_{(2)}^{\1}}\otimes^{\gamma}\tensor{h}{_{(2)}^{\2}}\otimes^{\gamma}\tensor{h}{_{(4)}^{\1}}\otimes1_{H_{\gamma}}u_{\gamma}(h_{(1)})u_{\gamma}(h_{(3)}),
			\end{split}
		\end{equation*}
		where, we once again used the properties of the translation map, the definition of $S_{\gamma}$, and the invertibility of $u_{\gamma}$.
		
		For the second part of the statement, we just need to check that the equality holds for $i=0$, and then it follows from induction that it is true for any other $i>0$. For any $h\in H_{\gamma}$, one has
		\begin{equation*}
			\begin{split}
				&d_0(c_n(\ell_{\gamma})(h))=\\
				&=(\tensor{h}{_{(2n+2)}^{\2}}\cdot_{\gamma}\tensor{h}{_{(2)}^{\1}})\cdot_{\gamma}\tensor{h}{_{(2)}^{\2}}\otimes^{\gamma}\dots\otimes^{\gamma}\tensor{h}{_{(2n)}^{\2}}\otimes^{\gamma}\tensor{h}{_{(2n+2)}^{\1}}u_{\gamma}(h_{(1)})\dots u_{\gamma}(h_{(2n+1)})\\
				&=\tensor{h}{_{(2n)}^{\2}}\cdot_{\gamma}(\tensor{h}{_{(2)}^{\1}}\cdot_{\gamma}\tensor{h}{_{(2)}^{\2}}u_{\gamma}(h_{(1)}))\otimes^{\gamma}\dots\otimes^{\gamma}\tensor{h}{_{(2n)}^{\2}}\otimes^{\gamma}\tensor{h}{_{(2n+2)}^{\1}}u_{\gamma}(h_{(3)})\dots u_{\gamma}(h_{(2n+1)})\\
				&=\tensor{h}{_{(2n)}^{\2}}\otimes^{\gamma}\tensor{h}{_{(2)}^{\1}}\otimes^{\gamma}\dots\otimes^{\gamma}\tensor{h}{_{(2n-2)}^{\2}}\otimes^{\gamma}\tensor{h}{_{(2n)}^{\1}}u_{\gamma}(h_{(1)})\dots u_{\gamma}(h_{(2n-1)})=c_{n-1}(\ell_{\gamma})(h),
			\end{split}
		\end{equation*}
		we used the definition of the counit and \ref{thm:deformed_strong_conn_right}.
	\end{proof}

	This lemma allows us to define a map $\widetilde{\mathrm{chw_n}}(\ell_{\gamma}):H^{tr}_{\gamma}\longrightarrow \mathrm{HC_{2n}}(M_{\gamma})$ in the same way as the non-deformed case, equation \eqref{eq:Chern-Weil_map}. Applying the counit $\underline{\epsilon}_{\gamma}^{\otimes(n+1)}$ after the map $c_n(\ell_{\gamma})$ we have the element $(\underline{\epsilon}_{\gamma}^{\otimes^{\gamma}(n+1)}\circ c_n(\ell_{\gamma}))(h)$ with $h\in H^{tr}_{\gamma}$    
	\begin{equation*}
		\begin{split}
			&\tensor{h}{_{(2n+2)}^{\2}}\cdot_{\gamma}\tensor{h}{_{(2n+2)}^{\2}}\otimes^{\gamma}\dots\otimes^{\gamma}\tensor{h}{_{(2n)}^{\2}}\cdot_{\gamma}\tensor{h}{_{(2n+2)}^{\1}}u_{\gamma}(h_{(1)})\dots u_{\gamma}(h_{(2n+1)})\\
			&=\tensor{h}{_{(5n+4)}^{\2}}\tensor{h}{_{(4)}^{\1}}\otimes^{\gamma}\dots\otimes^{\gamma}\tensor{h}{_{(5n-1)}^{\2}}\tensor{h}{_{(5n+4)}^{\1}}\\
			&\gamma^{-1}(h_{(5n+5)}, S(h_{(3)}))\dots\gamma^{-1}(h_{(5n)}, S(h_{(5n+3)}))\gamma(h_{(1)}, S(h_{(2)}))\dots \gamma(h_{(5n+1)}, S(h_{(5n+2)}))\\
			&=\tensor{h}{_{(n+4)}^{\2}}\tensor{h}{_{(4)}^{\1}}\otimes^{\gamma}\dots\otimes^{\gamma}\tensor{h}{_{(n+3)}^{\2}}\tensor{h}{_{(n+4)}^{\1}}\gamma^{-1}(h_{(n+5)}, S(h_{(3)}))\gamma(h_{(1)}, S(h_{(2)}))\\
			&=\tensor{h}{_{(n+3)}^{\2}}\tensor{h}{_{(3)}^{\1}}\otimes^{\gamma}\dots\otimes^{\gamma}\tensor{h}{_{(n+2)}^{\2}}\tensor{h}{_{(n+3)}^{\1}}(\gamma^{-1}*\gamma)(h_{(1)}, S(h_{(2)}))\\
			&=\tensor{h}{_{(n+1)}^{\2}}\tensor{h}{_{(1)}^{\1}}\otimes^{\gamma}\dots\otimes^{\gamma}\tensor{h}{_{(n)}^{\2}}\tensor{h}{_{(n+1)}^{\1}},
		\end{split}
	\end{equation*}
	that belongs to the algebra $B^{\otimes^{\gamma}(n+1)}$ and when we put it in the formula of the Chern-Weil homomorphism we get
	\begin{equation}
		\label{eq:Chern-Weil_hom_equality_right}
		\mathrm{chw}_n(\ell)=\mathrm{chw}_n(\ell_{\gamma}),
	\end{equation}
	In other words, the cyclic homology Chern-Weil homomorphism a is deformation invariant. Due to this phenomenon we have that the result of Proposition \ref{prop:pushforward_Chern-Weil} holds as well in this case, in other words any characteristic class is a deformation invariant for $2$-cocycles coming from the structure Hopf algebra.
	
	\subsection{Deformation from an external symmetry}
	In this case, we consider a $(K,H)$-bicomodule algebra such that $B\subset A$ is a principal $H$-comodule algebra and we take a $2$-cocycle of the external symmetry $\sigma:K\otimes K\longrightarrow\bK$. We saw that both $A$ and $B$ are deformed into $\tensor[_{\sigma}]{A}{}$ and $\tensor[_{\sigma}]{B}{}$, while the structure Hopf algebra $H$ remains the same.
	
	Like in the previous case, we have the linear isomorphism $M=A\square^H A\simeq \tensor[_{\sigma}]{M}{}=\tensor[_{\sigma}]{A}{}\square^H\tensor[_{\sigma}]{A}{}$. The counit on $\tensor[_{\sigma}]{M}{}$ is  
	\beq
	\label{eq:counit_bialgebroid_deformed_sigma}
	\tensor[_{\sigma}]{\underline{\epsilon}}{}:\tensor[_{\sigma}]{M}{}\longrightarrow\tensor[_{\sigma}]{B}{},\quad a\tensor[^{\sigma}]{\otimes}{}\tilde{a}\longmapsto\sigma\left(a_{(-1)}, \tilde{a}_{(-1)}\right)a_{(0)}\tilde{a}_{(0)},
	\eeq 
	thus the algebra structure is given by the multiplication
	\beq
	\label{eq:product_M_sigma}
	m\bullet_{\sigma}m':=\tensor[_{\sigma}]{\underline{\epsilon}}{}(m)m'.
	\eeq 
	We remark that in the latter formula, the product \eqref{eq:deform_prod_left_algebra} is used, so explicitly we have
	\begin{equation*}
		\mu_{\tensor[_{\sigma}]{M}{}}(a\tensor[^{\sigma}]{\otimes}{}\tilde{a}\otimes a'\tensor[^{\sigma}]{\otimes}{}\tilde{a}')=\tensor[_{\sigma}]{\underline{\epsilon}}{}(a\tensor[^{\sigma}]{\otimes}{}\tilde{a})\bullet_{\sigma}a'\tensor[^{\sigma}]{\otimes}{}\tilde{a}',
	\end{equation*}
	where now $\bullet_{\sigma}$ is the multiplication in $\tensor[_{\sigma}]{A}{}$.
	
	For a deformation via a $2$-cocycle $\sigma$ of an external Hopf algebra $K$ we have the following result
	\begin{prop}
		\label{prop:deforma_Chern-Weil_external}
		For any $n\in\mathbb{N}$ the map
		\begin{equation*}
			\begin{split}
				&c_n(\tensor[_{\sigma}]{\ell}{}):H^{tr}\longrightarrow \tensor[_{\sigma}]{M}{^{\otimes(n+1)}},\\
				h\longmapsto\sigma^{-1}(\tensor{h}{_{(1)}^{\1}_{(-1)}}&,\tensor{h}{_{(1)}^{\2}_{(-1)}}) \dots\sigma^{-1}(\tensor{h}{_{(n+1)}^{\1}_{(-1)}},\tensor{h}{_{(n+1)}^{\2}_{(-1)}})\\
				\tensor{h}{_{(n+1)}^{\2}_{(0)}}\tensor[^{\sigma}]{\otimes}{}&\tensor{h}{_{(1)}^{\1}_{(0)}}\tensor[^{\sigma}]{\otimes}{}\dots\tensor[^{\sigma}]{\otimes}{}\tensor{h}{_{(n)}^{\2}_{(0)}}\tensor[^{\sigma}]{\otimes}{}\tensor{h}{_{(n+1)}^{\1}_{(0)}}.    
			\end{split}
		\end{equation*}
		is well-defined and its image lies in the cyclic-symmetric part of $\tensor[_{\sigma}]{M}{^{\otimes(n+1)}}$. Moreover for any face operator $d_i$ with $i=0,\dots,n$ one has
		\begin{equation*}
			d_ic_n(_{\sigma}\ell)=c_{n-1}(_{\sigma}\ell)
		\end{equation*}
	\end{prop}
	\begin{proof}
		The proof goes in the same way as in \ref{prop:Chern-Weil_deform_right}. Let $\rho^{\otimes}$ be the diagonal coaction of $H$ on $_{\sigma}A$, then
		\begin{equation*}
			\begin{split}
				&(\rho^{\otimes}\tensor[^{\sigma}]{\otimes}{}\mathrm{id})(c_n(\tensor[_{\sigma}]{\ell}{})(h))=\\
				&=\sigma^{-1}\left(\tensor{h}{_{(1)}^{\1}_{(-1)}},\tensor{h}{_{(1)}^{\2}_{(-1)}}\right) \dots\sigma^{-1}\left(\tensor{h}{_{(n+1)}^{\1}_{(-1)}},\tensor{h}{_{(n+1)}^{\2}_{(-1)}}\right)\\
				&\tensor{h}{_{(n+1)}^{\2}_{(0)(0)}}\tensor[^{\sigma}]{\otimes}{}\tensor{h}{_{(1)}^{\1}_{(0)(0)}}\otimes\tensor{h}{_{(n+1)}^{\2}_{(0)(1)}}\tensor{h}{_{(1)}^{\1}_{(0)(1)}}\tensor[^{\sigma}]{\otimes}{}\dots\tensor[^{\sigma}]{\otimes}{}\tensor{h}{_{(n)}^{\2}_{(0)}}\tensor[^{\sigma}]{\otimes}{}\tensor{h}{_{(n+1)}^{\1}_{(0)}}\\
				&=\sigma^{-1}\left(\tensor{h}{_{(1)}^{\1}_{(0)(-1)}},\tensor{h}{_{(1)}^{\2}_{(-1)}}\right) \dots\sigma^{-1}\left(\tensor{h}{_{(n+1)}^{\1}_{(-1)}},\tensor{h}{_{(n+1)}^{\2}_{(0)(-1)}}\right)\\
				&\tensor{h}{_{(n+1)}^{\2}_{(0)(0)}}\tensor[^{\sigma}]{\otimes}{}\tensor{h}{_{(1)}^{\1}_{(0)(0)}}\otimes\tensor{h}{_{(n+1)}^{\2}_{(1)}}\tensor{h}{_{(1)}^{\1}_{(1)}}\otimes\dots\tensor[^{\sigma}]{\otimes}{}\tensor{h}{_{(n)}^{\2}_{(0)}}\tensor[^{\sigma}]{\otimes}{}\tensor{h}{_{(n+1)}^{\1}_{(0)}}\\
				&=\sigma^{-1}\left(\tensor{h}{_{(2)}^{\1}_{(-1)}},\tensor{h}{_{(2)}^{\2}_{(-1)}}\right) \dots\sigma^{-1}\left(\tensor{h}{_{(n+2)}^{\1}_{(-1)}},\tensor{h}{_{(n+2)}^{\2}_{(-1)}}\right)\\
				&\tensor{h}{_{(n+2)}^{\2}_{(0)}}\tensor[^{\sigma}]{\otimes}{}\tensor{h}{_{(2)}^{\1}_{(0)}}\otimes h_{(n+3)}S(h_{(1)})\otimes\dots\tensor[^{\sigma}]{\otimes}{}\tensor{h}{_{(n+1)}^{\2}_{(0)}}\tensor[^{\sigma}]{\otimes}{}\tensor{h}{_{(n+2)}^{\1}_{(0)}}\\
				&=\sigma^{-1}\left(\tensor{h}{_{(1)}^{\1}_{(-1)}},\tensor{h}{_{(1)}^{\2}_{(-1)}}\right) \dots\sigma^{-1}\left(\tensor{h}{_{(n+1)}^{\1}_{(-1)}},\tensor{h}{_{(n+1)}^{\2}_{(-1)}}\right)\\
				&\tensor{h}{_{(n+1)}^{\2}_{(0)}}\tensor[^{\sigma}]{\otimes}{}\tensor{h}{_{(1 )}^{\1}_{(0)}}\otimes h_{(n+2)}S(h_{(n+3)})\otimes\dots\tensor[^{\sigma}]{\otimes}{}\tensor{h}{_{(n)}^{\2}_{(0)}}\tensor[^{\sigma}]{\otimes}{}\tensor{h}{_{(n+1)}^{\1}_{(0)}}\\
				&=\sigma^{-1}\left(\tensor{h}{_{(1)}^{\1}_{(-1)}},\tensor{h}{_{(1)}^{\2}_{(-1)}}\right) \dots\sigma^{-1}\left(\tensor{h}{_{(n+1)}^{\1}_{(-1)}},\tensor{h}{_{(n+1)}^{\2}_{(-1)}}\right)\\
				&\tensor{h}{_{(n+1)}^{\2}_{(0)}}\tensor[^{\sigma}]{\otimes}{}\tensor{h}{_{(1)}^{\1}_{(0)}}\otimes 1_H\otimes\dots\tensor[^{\sigma}]{\otimes}{}\tensor{h}{_{(n)}^{\2}_{(0)}}\tensor[^{\sigma}]{\otimes}{}\tensor{h}{_{(n+1)}^{\1}_{(0)}},
			\end{split}
		\end{equation*}
		where we used the fact that $\tensor[_{\sigma}]{A}{}$ is a $K$-$H$-bicomodule, the properties of the strong connection and the cyclic property in $H^{tr}$.
		
		Now, if we apply the face operator $d_0$, what we get is
		\begin{equation*}
			\begin{split}
				&d_0(c_n(\tensor[_{\sigma}]{\ell}{})(h))=\\
				&=\sigma^{-1}\left(\tensor{h}{_{(1)}^{\1}_{(-1)}},\tensor{h}{_{(1)}^{\2}_{(-1)}}\right) \dots\sigma^{-1}\left(\tensor{h}{_{(n+1)}^{\1}_{(-1)}},\tensor{h}{_{(n+1)}^{\2}_{(-1)}}\right)\\
				&(\tensor{h}{_{(n+1)}^{\2}_{(0)}}\bullet_{\sigma}\tensor{h}{_{(1)}^{\1}_{(0)}})\bullet_{\sigma}\tensor{h}{_{(2)}^{\2}_{(0)}}\tensor[^{\sigma}]{\otimes}{}\dots\tensor[^{\sigma}]{\otimes}{}\tensor{h}{_{(n)}^{\2}_{(0)}}\tensor[^{\sigma}]{\otimes}{}\tensor{h}{_{(n+1)}^{\1}_{(0)}}\\
				&=\sigma^{-1}\left(\tensor{h}{_{(2)}^{\1}_{(-1)}},\tensor{h}{_{(2)}^{\2}_{(-1)}}\right)\dots\sigma^{-1}\left(\tensor{h}{_{(n+1)}^{\1}_{(-1)}},\tensor{h}{_{(n+1)}^{\2}_{(-1)}}\right)\\
				&\tensor{h}{_{(n+1)}^{\2}_{(0)}}\bullet_{\sigma}\left(\sigma^{-1}\left(\tensor{h}{_{(1)}^{\1}_{(-1)}},\tensor{h}{_{(1)}^{\2}_{(-1)}}\right)\tensor{h}{_{(1)}^{\1}}\bullet_{\sigma}\tensor{h}{_{(1)}^{\2}}\right)\tensor[^{\sigma}]{\otimes}{}\dots\tensor[^{\sigma}]{\otimes}{}\tensor{h}{_{(n)}^{\2}_{(0)}}\tensor[^{\sigma}]{\otimes}{}\tensor{h}{_{(n+1)}^{\1}_{(0)}}\\
				&=\sigma^{-1}\left(\tensor{h}{_{(1)}^{\1}_{(-1)}},\tensor{h}{_{(1)}^{\2}_{(-1)}}\right)\dots\sigma^{-1}\left(\tensor{h}{_{(n)}^{\1}_{(-1)}},\tensor{h}{_{(n)}^{\2}_{(-1)}}\right)\\
				&\tensor{h}{_{(n)}^{\2}_{(0)}}\tensor[^{\sigma}]{\otimes}{}\tensor{h}{_{(1)}^{\1}_{(0)}}\tensor[^{\sigma}]{\otimes}{}\dots\tensor[^{\sigma}]{\otimes}{}\tensor{h}{_{(n-1)}^{\2}_{(0)}}\tensor[^{\sigma}]{\otimes}{}\tensor{h}{_{(n)}^{\1}_{(0)}}=c_{n-1}(\tensor[_{\sigma}]{\ell}{})(h),
			\end{split}
		\end{equation*}
		where we used the associativity of the product $\bullet_{\sigma}$ and \ref{thm:strong_conn_left}
	\end{proof}
	
	The counit $\tensor[_{\sigma}]{\underline{\epsilon}}{}$, together with map $\widetilde{\mathrm{chw}}_n(\tensor[_{\sigma}]{\ell}{}):H^{tr}\longrightarrow\mathrm{HC}_{2n}(\tensor[_{\sigma}]{M}{})$, induces a map $\mathrm{chw}_n(\tensor[_{\sigma}]{\ell}{}):H^{tr}\longrightarrow\mathrm{HC}_{2n}(\tensor[_{\sigma}]{B}{})$ valued in the cyclic homology of the base algebra $\tensor[_{\sigma}]{B}{}$ that is defined starting with $x_n(\tensor[_{\sigma}]{\ell}{},h):=(\tensor[_{\sigma}]{\underline{\epsilon}}{}^{\tensor[_{\sigma}]{\otimes}{}(n+1)}\circ c_n(\tensor[_{\sigma}]{\ell}{}))(h)$
	\begin{equation}
		\label{eq:Chern-Weil_map_deformed_left}
		\begin{split}
			x_n(\tensor[_{\sigma}]{\ell}{},h)=&\sigma^{-1}\left(\tensor{h}{_{(1)}^{\1}_{(-2)}},\tensor{h}{_{(1)}^{\2}_{(-2)}}\right)\dots\sigma^{-1}\left(\tensor{h}{_{(n+1)}^{\1}_{(-2)}},\tensor{h}{_{(n+1)}^{\2}_{(-2)}}\right)\\
			&\sigma\left(\tensor{h}{_{(n+1)}^{\2}_{(-1)}},\tensor{h}{_{(1)}^{\1}_{(-1)}}\right)\dots\sigma\left(\tensor{h}{_{(n)}^{\2}_{(-1)}},\tensor{h}{_{(n+1)}^{\1}_{(-1)}}\right)\\
			&\tensor{h}{_{(n+1)}^{\2}_{(0)}}\tensor{h}{_{(1)}^{\1}_{(0)}}\tensor[^{\sigma}]{\otimes}{}\dots\tensor[^{\sigma}]{\otimes}{}\tensor{h}{_{(n)}^{\2}_{(0)}}\tensor{h}{_{(n+1)}^{\1}_{(0)}}.
		\end{split}
	\end{equation}
	So we have the formula
	\begin{equation}
		\label{eq:Chern-Weil_map_deformed_ext}
		\chw_n(\tensor[_{\sigma}]{\ell}{})(h)=\sum_{i=0}^{2n}(-1)^{\lfloor\frac{i}{2}\rfloor}\frac{i!}{\lfloor\frac{i}{2}\rfloor!}x_i(\tensor[_{\sigma}]{\ell}{},h),\quad h\in H^{tr}.
	\end{equation}
	
	We notice in this case that there is no cancellation of the cocycle $\sigma$ in the final formula. Thus, under a deformation of this type, we have the Chern-Weil homomorphism changes. 
	
	\subsection{Deformations combined}
	If we now take $2$-cocycles, namely $\gamma:H\otimes H\longrightarrow\bK$ and $\sigma:K\otimes K\longrightarrow\bK$ and perform a combined deformation of a $(K,H)$-bicomoule algebra $A$ such that $B\subseteq A$ is a principal $H$-comodule algebra, we saw that $\tensor[_{\sigma}]{B}{}\subseteq\tensor[_{\sigma}]{A}{_{\gamma}}$ is a principal $\tensor[]{H}{_{\gamma}}$-comodule algebra with external symmetry $\tensor[_{\sigma}]{K}{}$ and with strong connection $\tensor[_{\sigma}]{l}{_{\gamma}}$ of \eqref{eq:strong_conn_both_deformed}. For what we studied in this section, we find, by performing the same type of computations in \ref{prop:Chern-Weil_deform_right} and \ref{prop:deforma_Chern-Weil_external}, that the Chern-Weil homomorphism is subjected to the change only by the $2$-cocycle of the external symmetry $K$, i.e.
	\begin{equation*}
		\chw_n(\tensor[_{\sigma}]{\ell}{_{\gamma}})=\chw_n((\tensor[_{\sigma}]{\ell}{})_{\gamma})=\chw_n(\tensor[_{\sigma}]{\ell}{}),\quad n\in\bN.
	\end{equation*}
	
	%    Considering deformations via $2$-cocycles, we have that the same result holds. For a deformation of the structure Hopf algebra $H$ and the algebra $A$, we have that the Chern-Galois character does not change. In proposition \ref{prop:Chern-Weil_deform_right} and following discussion we saw that also the Chern-Weil homomorphism does not change, so Proposition $4.8$ of \cite{hajac2021cyclic} applies. If we deform the algebras $A$
	%	and $B$ with a $2$-cocycle of an external symmetry $K$, one can show, with the same type of computations performed in the proof of \ref{prop:deforma_Chern-Weil_external}, that the Chern-Galois character makes the following diagram commute
	%	\begin{equation*}
	%		\begin{tikzcd}
	%			\mathrm{Corep}_f(H) \arrow[r, "\Ass"] \arrow[d, "\chi"'] & K_0(\tensor[_{\sigma}]{B}{}) \arrow[d, "\mathrm{ch}_n"]\\
	%			H^{tr} \arrow[r, "\chw_n(\ls)"] & HC_{2n}(\tensor[_{\sigma}]{B}{}),
	%		\end{tikzcd}
	%	\end{equation*}
	%	thus, the same result holds.

	\subsection{Naturality and deformations}
	If we consider a deformation via a $2$-cocycle of the structure algebra $H$, any right $H$-comodule algebra morphism $f:A\longrightarrow\bar{A}$ satisfies
	\begin{equation}
		\label{eq:algebra_map_def}
		\begin{split}
			f(a)\cdot_{\gamma} f(a')&=f(a)_{(0)}f(a')_{(0)}\gamma^{-1}(f(a)_{(1)},f(a')_{(1)})\\
			&=f(a_{(0)})f(a'_{(0)})\gamma^{-1}(a_{(1)},a'_{(1)})\\
			&=f(a_{(0)}a'_{(0)})\gamma^{-1}(a_{(1)},a'_{(1)}))=f(a\cdot_{\gamma} a'), \quad a,a'\in A_{\gamma},
		\end{split}
	\end{equation}
	it is then a $H_{\gamma}$-comodule algebra morphism between the deformaed algebras $A_{\gamma}$ and $\bar{A}_{\gamma}$. For this reason, we also have the identification of principal comodule algebra categories $\cQ b_{H_{\gamma}}(B)\simeq\cQ b_H(B)$. Consequently, the claim in Proposition \ref{prop:pushforward_Chern-Weil} and the discussion on the cleft case after it still hold in this case.
	
	For a deformation via a $2$-cocycle $\sigma$ of an external symmetry $K$ of a principal $H$-comodule algebra, to have the same result, it is sufficient to ask that $f:A\longrightarrow\bar{A}$ is a $(K,H)$-bicomodule algebra morphism. In this way, we have an induced map from $\tensor[_{\sigma}]{A}{}$ to $\tensor[_{\sigma}]{\bar{A}}{}$ because of \eqref{eq:algebra_map_def}. If we denote by $\tensor[^{K}]{\cQ b}{_{H}}(B)$ the category of isomorphism classes of principal $H$-comodule algebras with external symmetry $K$, we have the following
	\begin{prop}
		\label{prop:deformed_pullback}
		The deformation via $\sigma$ of the map \eqref{eq:characteristic_class} $c_{\tensor[_{\sigma}]{B}{}}:\tensor[^{\tensor[_{\sigma}]{K}{}}]{\cQ b}{_{H}}(\tensor[_{\sigma}]{B}{})\longrightarrow HC(\tensor[_{\sigma}]{B}{})$ induced by the deformed Chern-Weil homomorphism of \eqref{eq:Chern-Weil_map_deformed_left} is a natural transformation between the functors $\tensor[^{\tensor[_{\sigma}]{K}{}}]{\cQ b}{_{H}}(-)$ and $HC(-)$.
	\end{prop}
	\begin{proof}
		The proof follows from that of \ref{prop:pushforward_Chern-Weil}.
	\end{proof}
	
	Finally, combining this last result and the fact that $\chw(\tensor[_{\sigma}]{\ell}{_{\gamma}})=\chw(\tensor[_{\sigma}]{\ell}{})$, we have that Proposition \ref{prop:pushforward_Chern-Weil} is still valid for the combined deformation case.
	%\addcontentsline{toc}{section}{References}
	
	\bigskip
	\noindent{\bf Acknowledgements.}
	
	I thank Piotr M. Hajac and Chiara Pagani for fruitful discussions. At the beginning of this work, I was partially supported by the EU Horizon Project "Graph Algebras" Grant agreement ID: 101086394. 
	
	\bibliography{bib}{}
	\bibliographystyle{plain} 
\end{document}